\documentclass[a4paper,12pt]{article}

\usepackage[margin=1.2in]{geometry}

\usepackage{amsmath}
\usepackage{algorithmic}
\usepackage{algorithm}
\usepackage{amsthm}
\usepackage{amsfonts}
\usepackage{comment}
\usepackage{graphicx}
\usepackage{hyperref}
\usepackage{color}
\usepackage{subcaption}

\usepackage{array,multirow}
\usepackage{array}

\newtheorem{assumption} {Assumption}
\newtheorem{theorem} {Theorem}
\newtheorem{lemma} {Lemma}

\newtheorem{observation} {Observation}

\def\d{{\mathbf{d}}}
\def\c{{\mathbf{c}}}

\def\a{{\mathbf{a}}}
\def\u{{\mathbf{u}}}
\def\v{{\mathbf{v}}}
\def\z{{\mathbf{z}}}
\def\w{{\mathbf{w}}}
\def\k{{\mathbf{k}}}
\def\h{{\mathbf{h}}}
\def\y{{\mathbf{y}}}
\def\q{{\mathbf{q}}}
\def\p{{\mathbf{p}}}
\def\b{{\mathbf{b}}}
\def\bS{{\mathbf{S}}}
\def\X{{\mathbf{X}}}
\def\Y{{\mathbf{Y}}}
\def\A{{\mathbf{A}}}
\def\M{{\mathbf{M}}}
\def\I{{\mathbf{I}}}
\def\B{{\mathbf{B}}}

\def\V{{\mathbf{V}}}

\def\W{{\mathbf{W}}}
\def\U{{\mathbf{U}}}

\def\P{{\mathbf{P}}}

\newcommand{\mA}{\mathcal{A}}

\newcommand{\mX}{\mathcal{X}}
\newcommand{\mV}{\mathcal{V}}

\newcommand{\mK}{\mathcal{K}}
\newcommand{\mS}{\mathcal{S}}

\newcommand{\mbS}{\mathbb{S}}

\newcommand{\E}{\mathbb{E}}

\newcommand{\trace}{\textrm{Tr}}
\newcommand{\rank}{\textrm{rank}}
\newcommand{\reals}{\mathbb{R}}

\title{A Randomized Linearly Convergent Frank-Wolfe-type Method for Smooth Convex Minimization over the Spectrahedron}
\date{}
\author{Dan Garber \\\ {\small{Technion - Israel Institute of Technology}}\\  {\small{dangar@technion.ac.il}} }

\begin{document}
\maketitle

%\abstract{
\begin{abstract}
We consider the problem of minimizing a smooth and convex function over the $n$-dimensional spectrahedron --- the set of real symmetric $n\times n$ positive semidefinite matrices with unit trace, which underlies numerous applications in statistics, machine learning and additional domains. Standard first-order methods often require high-rank matrix computations which are prohibitive when the dimension $n$ is large. The well-known Frank-Wolfe method on the other hand only requires efficient rank-one matrix computations, however, suffers from worst-case slow convergence, even under conditions that enable linear convergence rates for standard methods. In  this work we  present the first Frank-Wolfe-based algorithm that only applies efficient rank-one matrix computations and, assuming quadratic growth and strict complementarity conditions, is guaranteed, after a finite number of iterations, to converge linearly, in expectation, and independently of the ambient dimension.
\end{abstract}

\section{Introduction}
This paper is concerned with efficient first-order optimization methods (suitable for high dimensional settings) for minimizing a smooth (Lipschitz continuous gradient) and convex objective function $f:\reals^{n\times n}\rightarrow \reals$ over the $n$-dimensional spectrahedron --- the set of all $n\times n$ real symmetric positive semidefinite matrices with unit trace $\mS^n := \{\X\in\mbS^n~|~\X\succeq 0, \trace(\X)=1\}$, where $\mbS^n$ denotes the space of real symmetric $n\times n$ matrices, and $\X\succeq0$ denotes that $\X$ is positive semidefinite. This problem underlies numerous applications of interest in diverse fields of science and engineering such as statistics, machine learning, and discrete optimization, including for instance  convex relaxations to low-rank matrix recovery problems \cite{candes2010matrix, srebro2004maximum, candes2015phase, candes2011robust, jaggi2010simple, tropp2015convex}, covariance matrix estimation problems \cite{wiesel2012geodesic, sun2016robust, danon2022frank}, convex relaxations to hard combinatorial problems \cite{yurtsever2019conditional, gidel2018frank, pham2023scalable}, and many more. More specifically, we are interested in methods inspired by the well-known Frank-Wolfe method (FW), aka the conditional gradient method \cite{frank1956algorithm, levitin1966constrained, jaggi2013revisiting}. These types of methods, while applicable more broadly to smooth convex minimization over arbitrary convex and compact subsets of an inner product space, are in particular interesting in case the feasible set is the spectrahedron, see for instance \cite{hazan2008sparse, jaggi2010simple, NIPS2016_df877f38, yurtsever2017sketchy, freund2017extended} \footnote{some papers, e.g.,  \cite{jaggi2010simple, freund2017extended}, consider the closely related problem of optimization over the matrix nuclear norm ball, see  \cite{jaggi2010simple} for a discussion of the connection between these setups}, for a number of reasons. First, while orthogonal projections onto the spectrahedron, that are applied in  standard projection-based gradient methods, require in worst case a full eigen-decomposition of a symmetric $n\times n$ matrix, which in practical implementations runs in $O(n^3)$ time, a linear optimization step over the spectrahedron as required by FW, amounts only to an efficient rank-one matrix computation: a single leading eigenvector computation of a symmetric $n\times n$ matrix, which often runs in nearly linear time in the size of the matrix using fast iterative eigenvector methods. Second, since FW only applies a rank-one update, in many cases of interest both the objective function value and gradient direction could be updated from one iteration to the next much more efficiently than when the update is of high rank. This is for instance the case when the objective $f$ includes terms such as $g(\mA\X)$ (e.g., in matrix sensing problems  \cite{yurtsever2017sketchy})  or $g(\mA\X^{-1})$ (e.g., in covariance estimation problems  \cite{danon2022frank}), where $\mA:\reals^{n\times n}\rightarrow\reals^m$ is a linear map composed of rank-one matrices: $\mA\X = (\a_1^{\top}\X\b_1,\dots,\a_m^{\top}\X\b_m)^{\top}\in\reals^m$ with $\{(\a_i,\b_i)\}_{i=1}^m\subset\reals^n\times\reals^n$, and $g:\reals^m\rightarrow\reals$ is coordinate-wise separable. Another example is when the objective function includes a $\log\det(\X)$ component.  Third, another consequence of the low-rank updates of FW is that when initialized with a low-rank matrix, and provided that the number of iterations is substantially lower than the dimension $n$, the iterates produced by the method are low-rank and can be stored in factorized form to further reduce memory and runtime requirements. Finally, FW-based methods are often parameter-free and do not require tuning of any parameters (such as the smoothness or strong convexity parameters), and can often be implemented using simple and efficient line-search.

The major drawback of Frank-Wolfe, however, is worst-case slow convergence rates. For minimizing a $\beta$-smooth convex objective over the spectrahedron, the worst-case number of iterations to guarantee an $\epsilon$-approximated solution in function value is $O(\beta/\epsilon)$ \cite{hazan2008sparse, jaggi2010simple, jaggi2013revisiting}. This does not improve even under standard curvature assumptions such as quadratic growth (see for instance lower bound in \cite{jaggi2011convex}), which are well-known to facilitate linear  convergence rates (i.e., convergence rates that scale only with $\log(1/\epsilon)$) for standard projection-based methods \cite{drusvyatskiy2018error, necoara2019linear}.

In recent years, several Frank-Wolfe-inspired methods for optimization over the spectrahedron have been developed that, under conditions such as quadratic growth and/or strict complementarity, enjoy a linear convergence rate. \cite{garber2023linear} established that in case there exists a rank-one optimal solution which satisfies the strict complementarity condition, the standard FW with line-search, after a finite number of iterations, begins to converge linearly. This result, however, does not extend to the general case in which there is an optimal solution with rank greater than one.  The works \cite{allen2017linear, ding2020spectral} proposed methods, which we refer to as block-Frank-Wolfe methods,  that converge linearly without restriction on the rank of optimal solutions, however these  diverge significantly from the classical FW template: while standard FW only applies rank-one leading eigenvector computations,  \cite{allen2017linear, ding2020spectral} require on each iteration to compute the $r$ leading components in the  singular value decomposition (SVD) of an $n\times n$ matrix, where the rank parameter $r$ is required to be at least as large as the highest rank of any optimal solution. As a result, these methods have several inherent and substantial limitations. First, they require knowledge of the rank of optimal solutions which is often unknown and overestimating this quantity results in needlessly expensive SVD computations. Second, these methods are only efficient when  optimal solutions are of low-rank, since otherwise the partial SVD computations  required are not significantly more efficient than those in projection-based methods. Similarly, in this case also the important FW feature of fast update times of objective function and gradient information is lost. Finally, while FW can often be implemented with simple line-search, the method in \cite{allen2017linear}, aside from the smoothness constant $\beta$, requires the quadratic growth constant $\alpha$ in order to set the step-size, which is often very difficult to estimate. The method in  \cite{ding2020spectral} does not require these constants and instead performs a sophisticated multivariate ``step-size'' optimization on each iteration, by minimizing the original objective function $f$ over an $r$-dimensional spectrahedron (where $r$ is  the upper-bound on the rank of optimal solutions), for instance by running a nested first-order method. However, this may again only be efficient for small values of $r$.   
    
We mention in passing that aside from linear convergence results, the work \cite{NIPS2016_df877f38} presented a FW variant for the spectrahedron that also uses only rank-one matrix computations (leading eigenvector) and, under strong convexity of the objective function, is guaranteed to converge with a fast rate of the form $\min\{\frac{\rank(\X^*)^{2/3}}{t^{4/3}}, \frac{1}{\lambda_{\min}(\X^*)^2t^2}\}$, where $t$ is the iteration counter and $\lambda_{\min}(\X^*)$ denotes the smallest non-zero eigenvalue of the unique optimal solution $\X^*$. While this result does not limit the rank of the optimal solution and clearly improves upon the standard FW rate of $O(1/t)$, it is not a linear convergence rate.  
  
The dichotomy in linear convergence rates between the case of a rank-one optimal solution that can be solved via the standard FW method using only efficient rank-one matrix computations (under strict complementarity condition), and the case of a higher rank optimal solution which currently can only be  handled via block-Frank-Wolfe methods that require higher rank matrix computations (with all of the accompanying limitations discussed above) leads us to the following conceptual question:

\vspace{5pt}
\begin{center}
\textit{Are SVD computations of rank greater than one mandatory (in Frank-Wolfe-based methods) in order to guarantee a linear convergence rate (without restricting the rank of optimal solutions)?} 
\end{center}
\vspace{5pt}

%\begin{center}
%\textit{Is there a first-order method that: I. applies only efficient rank-one matrix computations, II. (almost) does not require difficult parameter tuning, and  III. under quadratic growth and strict complementarity conditions, converges with a (dimension-independent) linear rate?}
%\end{center}

We answer this question in the negative, under the assumption that quadratic growth and strict complementarity conditions hold (see formal definitions in Section \ref{sec:notAndAss}). We propose a novel Frank-Wolfe-based method that is guaranteed, after a finite number of iterations (``burn-in'' phase), to converge linearly, in expectation (the expectation is due to the use of randomization in one of the steps of the algorithm). In particular, both the finite number of initial steps and the linear convergence rate are independent of the ambient dimension.
Importantly, our  method admits an implementation, which aside from three efficient leading eigenvector computations  (the same type of rank-one computations applied in the standard FW method and can be executed in parallel), requires only $O(n^2)$ runtime per iteration. In terms of parameter tuning, our method requires only knowledge of the smoothness constant of the objective function $\beta$ (which as opposed to the rank of  optimal solutions or the quadratic growth constant, is often relatively easy to bound).\\

\begin{table*}\renewcommand{\arraystretch}{1}
{\small
\begin{center}

\begin{tabular}{ | p{9em} | p{2em} | p{8.2em}| p{3.4em}  | p{9.8em}  |} 
  \hline
  Algorithm &  SVD rank & req. parameters &burn-in phase? &  convergence rate  (after burn-in phase) \\
  \hline
  Frank-Wolfe \cite{jaggi2013revisiting} & $1$ & - & NO& $1/t$ \\ \hline
  Regularized-FW \cite{NIPS2016_df877f38} & 1 & $\beta$ & NO & $1/t^2$ (in expectation) \\ \hline
  Block-FW  \cite{allen2017linear}& $r $ & $r \geq \rank(\X^*), \alpha/\beta$ & NO &linear \\ \hline
  Spectral-FW \cite{ding2020spectral} & $r $ & $r \geq \rank(\X^*)$ & YES &linear \\ \hline
  This paper & $1$ & $\beta$ &YES & linear (in expectation) \\ \hline  
\end{tabular}\caption{Comparison of Frank-Wolfe-based methods for the spectrahedron.}\label{table:res}
\end{center}}
\vskip -0.2in
\end{table*}\renewcommand{\arraystretch}{1}  

%We answer this question in the affirmative, with  two slight reservations:  I. our proposed method does require the knowledge of a  single parameter:   the smoothness parameter of $f$ (which as opposed to the rank of the optimal solutions or the quadratic growth parameter, is often easy to bound), and II. As opposed to most standard methods, some steps in our proposed method involve the use of randomization. More concretely, our method is guaranteed, after a finite number of iterations, to converge linearly, in expectation. 
%Importantly, our  method admits an implementation, which aside of three efficient leading eigenvector computations of $n\times n$ symmetric matrices (which can be executed in parallel), requires only $O(n^2)$ runtime per iteration. 

Before moving on, we pause to discuss the strict complementarity assumption in some more detail. While such a condition does not typically appear in linear convergence results for projection-based methods, it was in fact highly instrumental to obtaining linear convergence rates for FW-based methods in a broader context: It was introduced in the classical text \cite{levitin1966constrained} for optimization over strongly convex sets, in \cite{guelat1986some, garber2020revisiting} for obtaining dimension-independent linear rates for polytopes, and in \cite{garber2023linear, ding2020spectral} for the spectrahedron setting. Moreover, \cite{ding2020spectral} proved that in our spectrahedron setting with an objective function which admits the highly popular structure $f(\X) = g(\mA\X)$, where $\mA$ is a linear map and $g$ is smooth and strongly convex,  strict complementarity further implies that $f$ satisfies quadratic growth over the spectrahedron.  \cite{zhou2017unified} further provided numerical evidence for a closely related setting (though not completely identical to ours),  that  for an objective function with the above structure, linear convergence may be unattainable when strict complementarity does not hold.

\subsection{Inspiration from the polytope setting}\label{sec:insp}
Aside from our matrix spectrahedron setting, the Frank-Wolfe method has gained significant interest in the past decades also for optimization over convex and compact polytopes. While in this setting also the worst-case convergence of the standard method (with line-search or without) scales with $1/\epsilon$, even under favorable conditions such as quadratic growth or strong convexity of the objective function, it was established in a series of works that simple modifications of the method, all of which rely in some way or another on the concept of incorporating  \textit{away steps} into the algorithm,  linear convergence can be established, and in particular that a linear rate independent of the ambient dimension could be established when strict complementarity also holds. In the celebrated work  \cite{guelat1986some}, which was the first to obtain a linear convergence result for polytopes (under strong convexity and strict complementarity), the authors proved that after a finite number of steps, all iterates of the method must lie inside the optimal face of the polytope, and there it acts as if performing unconstrained minimization and thus, once inside the optimal face, the strong convexity assumption implies linear convergence. 

Our method and analysis are highly inspired by \cite{guelat1986some}, but with some considerable differences. The main difficulty is that while polytopes have a finite number of faces which indeed allows to argue that under strict complementarity, from some iteration, all iterates  will lie inside the optimal face, such an argument is not sensible for the spectrahedron which has infinitely many faces\footnote{A face of $\mS^n$ is given by the set $\{\V\bS\V^{\top}~|~\bS\in\mS^r\}$ for some integer $r\in\{1,\dots,n\}$ and matrix $\V\in\reals^{n\times r}$ such that $\V^{\top}\V = \I$}. To bypass this difficulty, we apply our arguments w.r.t. a time-changing face of the spectrahedron, one that is defined according to a principal subspace of the current gradient direction. Once close enough to an optimal solution (which takes a finite number of steps), we can distinguish between two main cases: if our current iterate is sufficiently  aligned with this face, we can apply similar arguments to those in \cite{guelat1986some} (i.e., essentially unconstrained minimization) and argue that either a standard Frank-Wolfe step or a specialized away-step (designed specifically for our spectrahedron setting)  will reduce the approximation error by a constant fraction. Otherwise (in case the current iterate is not sufficiently aligned with the current face), we establish that a specialized \textit{randomized pairwise step}, i.e., a step that replaces a random rank-one component supported by the current iterate with a new rank-one component, reduces the error by a constant fraction, in expectation.  Thus, an overall linear convergence rate (in expectation) is obtained.

\subsection{Notation and Assumptions}\label{sec:notAndAss}
We denote matrices by uppercase boldface letters, e.g., $\X$, column vectors by lowercase boldface letters, e.g., $\v$, and scalars are denoted by lightface letters, e.g., $r,\alpha,C$.  For matrices we let $\Vert{\cdot}\Vert$ denote the spectral norm (largest singular value) and for column vectors we let it denote the Euclidean norm. We let $\Vert{\cdot}\Vert_F$ denote the Frobenius (Euclidean) norm for matrices. For a matrix $\A\in\mbS^n$ we let $\lambda_1(\A),\dots,\lambda_n(\A)$ denote its (real) eigenvalues in non-ascending order. We denote by $\textrm{Im}(\A)$ and $\textrm{Ker}(\A)$ its image and kernel, respectively, and by $\A^{\dagger}$ its Moore-Penrose pseudoinverse. We let $\I$ denote the identity matrix whose dimension will be clear from context. We let $\mbS^n_+$ denote the positive semidefinite cone in $\mbS^n$. We let $\langle{\A,\B}\rangle = \trace(\A^{\top}\B)$ denote the standard matrix inner-product.\\

We let $\mX^*\subseteq\mS^n$ denote the set of minimizers of the objective function $f$ over the spectrahedron $\mS^n$, and we let $f^*$ and $\nabla{}f^*$ denote the minimal value of $f$ over $\mS^n$ and the gradient direction at any minimizer (note that since $f$ is differentiable, the gradient direction is constant over $\mX^*$).

Recall we assume that $f$ is convex, and we also assume it is $\beta$-smooth over $\mS^n$, i.e., there exists $\beta \geq 0$ such that for any $\X,\Y\in\mS^n$ we have $\Vert{\nabla{}f(\X) - \nabla{}f(\Y)}\Vert_F \leq \beta\Vert{\X-\Y}\Vert_F$. Recall this further implies the well-known inequality (see for instance Lemma 5.7 in \cite{beck2017first}):
\begin{align*}
\forall (\X,\Y)\in\mS^n\times\mS^n: \quad f(\Y) \leq f(\X) + \langle{\Y-\X, \nabla{}f(\X)}\rangle + \frac{\beta}{2}\Vert{\Y-\X}\Vert_F^2,
\end{align*}
which we will use extensively. 

We now formally present our two main assumptions that will be assumed to hold throughout this work. 

Quadratic growth is a standard assumption in the literature on linear convergence rates for first-order methods \cite{drusvyatskiy2018error, necoara2019linear}.

\begin{assumption}[quadratic growth]\label{ass:qg}
There exists a scalar $\alpha > 0$ such that for any $\X\in\mS^n$,
\begin{align}\label{eq:QG}
\min_{\Y\in\mX^*}\Vert{\X-\Y}\Vert_F^2 \leq \frac{2}{\alpha}\left({f(\X) - f^*}\right).
\end{align}
\end{assumption}

Our second assumption is that \textit{all} optimal solutions satisfy the strict complementarity condition, which as discussed in the Introduction, has played a central role in proving linear convergence rates for Frank-Wolfe-type methods in diverse settings. In the case of the spectrahedron this takes the form of a positive eigengap in the gradient direction at optimal solutions (see also \cite{ding2020spectral, garber2023linear, garber2021convergence}) and that all optimal solutions have the same rank.

\begin{assumption}[strict complementarity of optimal solutions]\label{ass:sc}
There exists an integer $r^*\in\{1,\dots,n\}$ such that for all $\X^*\in\mX^*$ it holds that $\rank(\X^*) = r^*$,
%\begin{align}
%\min_{\X^*\in\mX^*}\lambda_{r^*}(\X^*) = \lambda_{r^*},
%\end{align}
and if $r^* < n$, then there exists a scalar $\delta  > 0$ such that
\begin{align}
\lambda_{n-r^*}(\nabla{}f^*) - \lambda_{n-r^*+1}(\nabla{}f^*) = \delta.  
%\min_{\X^*\in\mX^*}\left({\lambda_{n-r^*}(\nabla{}f(\X^*)) - \lambda_{n-r^*+1}(\nabla{}f(\X^*))}\right) = \delta.  
\end{align}
\end{assumption}

Following Assumption \ref{ass:sc} we denote the minimal value of the smallest non-zero eigenvalue over the optimal set:
\begin{align}
\lambda_{r^*} := \min_{\X^*\in\mX^*}\lambda_{r^*}(\X^*), 
\end{align}
which is guaranteed to be strictly positive.

Note Assumption \ref{ass:sc} implies that there exists  $\V^*\in\reals^{n\times r^*}$ such that $\V^{*\top}\V^* = \I$ and $\mX^*\subset \{\V^*\bS\V^{*\top}~|~\bS\in\mS^{r^*}, \bS\succ 0\}$, that is, $\mX^*$ lies in the interior of a $r^*$-dimensional face of $\mS^n$. This follows from the first-order optimality condition, see for instance Lemma 5.2 in \cite{garber2021convergence}.

One may wonder if in order to obtain our results it is indeed required, in case the optimal solution is not unique, that all of them satisfy  strict complementarity (which mandates that all solutions have the same rank). In some previous analyses of Frank-Wolfe methods with strict complementarity conditions \cite{guelat1986some, garber2020revisiting, garber2023linear} it was assumed that the optimal solution is unique (e.g., by assuming $f$ is strongly convex). Our analysis, which does not assume uniqueness, indeed requires that all solutions satisfy this condition. It remains open whether this is mandatory. 

%There exists an optimal solution $\X^*$ of rank $r^* = \rank(\X^*)$ which satisfies the strict complementarity condition, that is, if $r^* <n$ then,
%\begin{align}\label{eq:SC}
%\delta: = \lambda_{n-r^*}(\nabla{}f(\X^*)) - \lambda_{n-r^*+1}(\nabla{}f(\X^*)) > 0.
%\end{align}
%Moreover, the following non-degeneracy condition holds: writing the eigen-decomposition of $\X^*$ as $\X^* = \V^*\Lambda^*\V^{*\top}$ with $\V^*\in\reals^{n\times r^*}$ such that $\V^{*\top}\V^* = \I$, and a diagonal positive definite $r^*\times r^*$ matrix $\Lambda^*$, we have that the set of all optimal solution $\mX^*$, satisfies:
%\begin{align}\label{eq:nondegenerate}
%\mX^*\subseteq\{\V^*\M\V^{*\top} ~|~\M\in\mS^{r^*}, ~ \M \succeq \lambda_{r^*}\I\},
%\end{align}
%for some scalar $\lambda_{r^*} > 0$.

\subsection{Organization of this paper}
Section \ref{sec:algorithm} presents our algorithm and its efficient implementation. Section \ref{sec:analysis} then formally presents and proves the convergence rates of our proposed algorithm. Section \ref{sec:numerics} presents some numerical experiments. Finally, Section \ref{sec:future} outlines some open questions for future research.

\section{The Algorithm}\label{sec:algorithm}
Our algorithm is given below as Algorithm \ref{alg:AFW}. As discussed in Section \ref{sec:insp}, our algorithm is inspired by the \textit{Away-steps Frank-Wolfe} algorithm for optimization over polytopes \cite{guelat1986some, lacoste2015global, garber2020revisiting}.
Our algorithm applies three different types of steps: I. Standard Frank-Wolfe steps, II. Away/Drop steps, and III. Pairwise steps, all of which are detailed below. In Section \ref{sec:implement} we discuss in greater detail the efficient implementation of these steps. In particular we show how each step corresponds to a simple leading eigenvector computation. We also describe an implementation that, aside from these eigenvector computations (which can be computed in parallel), only requires $O(n^2)$ time per iteration. 

\paragraph*{Frank-Wolfe steps:} These  are based on moving from the current feasible point towards an extreme point of the feasible set that minimizes the inner-product with the current gradient direction. In  case of the spectrahedron, these correspond to a rank-one update of the form $\X_{t+1} \gets (1-\eta_t)\X_t + \eta_t\v_{t,+}\v_{t,+}^{\top}$, where $\v_{t,+}$ is a unit-length eigenvector corresponding to the leading eigenvalue of $-\nabla{}f(\X_t)$ (see also \cite{hazan2008sparse, jaggi2010simple, jaggi2013revisiting}). We set the step size $\eta_t\in[0,1]$ using line-search. 

\paragraph*{Away and Drop steps:} Away steps perform rank-one updates in which the weight of some rank-one component supported by the current iterate is lowered in favor of other components. This takes the form: $\X_{t+1} \gets \frac{1}{1-\eta_t}\left({\X_t - \eta_t\v_{t,-}\v_{t,-}^{\top}}\right)$, where $\v_{t,-}$ is a unit vector in $\textrm{Im}(\X_t)$ that maximizes the quadratic product with the current gradient direction. The step size $\eta_t$ is set via line search and satisfies $\eta_t\in[0,(\v_{t,-}^{\top}\X_t^{\dagger}\v_{t,-})^{-1}]$ in order to maintain the feasibility of $\X_{t+1}$ w.r.t. $\mS^n$ (see Lemma \ref{lem:stepsize} in the sequel). In particular, when $\eta_t = (\v_{t,-}^{\top}\X_t^{\dagger}\v_{t,-})^{-1}$ (i.e., the maximal step size allowed), we have that $\rank(\X_{t+1}) < \rank(\X_t)$, and then we refer to this as a \textit{drop step}. It should be mentioned that this type of away steps closely resembles the ones proposed in \cite{freund2017extended} for optimization over the matrix nuclear norm ball.
%%In particular, in order to encourage fast adaptation the iterates to the rank of the optimal face, on each iteration of Algorithm we first examine whether  a drop step will not increase the function value. If this is the case, we execute the drop step and continue to the next iteration. 

\paragraph*{Pairwise steps:} 
Pairwise steps correspond to rank-two updates, in which a rank-one component in the support of the current iterate is replaced with a new rank-one component. The rank-one component to be removed is chosen uniformly at random, and the new rank-one component is chosen according to a \textit{proximal gradient style} rule. Formally, the update takes the form: $\X_{t+1} \gets \X_t + \gamma_t(\u_{t,+}\u_{t,+}^{\top}-\u_{t,-}\u_{t,-}^{\top})$, where $\u_{t,-}$ (the component to be removed) is chosen uniformly at random from the unit sphere in $\textrm{Im}(\X_t)$, the step-size is set to $\gamma_t = \left({\u_{t,-}^{\top}\X_t^{\dagger}\u_{t,-}}\right)^{-1}$ (so that feasibility of $\X_{t+1}$ is guaranteed, see Lemma \ref{lem:stepsize}), and $\u_{t,+}$ is a unit vector chosen according to the rule:  
\begin{align}\label{eq:pws}
\u_{t,+} \gets \textrm{argmin}_{\Vert{\u}\Vert=1}\u^{\top}\nabla{}f(\X_t)\u + \frac{\beta\gamma_t}{2}\Vert{\u\u^{\top}-\u_{t,-}\u_{t,-}^{\top}}\Vert_F^2.
\end{align}
Note that since $\u$ in the RHS is constrained to have unit norm, the above simplifies to 
\begin{align*}
\u_{t,+} \gets \textrm{argmin}_{\Vert{\u}\Vert=1}\u^{\top}\left({\nabla{}f(\X_t)-\beta\gamma_t\u_{t,-}\u_{t,-}^{\top}}\right)\u,
\end{align*}
which in turn implies that $\u_{t,+}$ is simply a unit-length leading eigenvector of the matrix $\beta\gamma_t\u_{t,-}\u_{t,-}^{\top}-\nabla{}f(\X_t)$. 
Note that as opposed to previous steps, this step involves both the knowledge of the smoothness parameter $\beta$ and the use of randomization. Note furthermore that in this step we always take the maximal step-size $\gamma_t$ that is guaranteed to maintain feasibility (see Lemma \ref{lem:stepsize}). In some sense,  one can think about the computation of $\u_{t,+}$ in Eq. \eqref{eq:pws} as already incorporating the choice of step-size within the choice of $\u_{t,+}$ (the interested reader can formally see  this in the proof  of Lemma \ref{lem:pfwstep}, which analyzes the benefit of applying these steps).

\paragraph*{Choice of step:} On each iteration, our algorithm first attempts to perform a drop step (i.e., an away step with maximal step-size) in order to quickly adapt the rank of the iterates to that of the optimal face. If this results in a non-increasing  objective value, then the algorithm continues to the next iteration. Otherwise, it computes all three possibilities (can be done in parallel): Frank-Wolfe step,  Away step, and a Pairwise step. The algorithm then follows the step which reduces the objective value the most.

\begin{algorithm}[H]
\begin{algorithmic}
\caption{Spectrahedron Frank-Wolfe with Away and Pairwise steps}\label{alg:AFW}
\STATE $\X_1 \gets $ arbitrary matrix in $\mS^n$
\FOR{$t=1,2,\dots$}
\STATE $\v_{t,-} \gets\textrm{argmax}_{\v\in\textrm{Im}(\X_t):\Vert{\v}\Vert=1}\v^{\top}\nabla{}f(\X_t)\v$
%\STATE $\v_{t,-} \gets\arg\max_{\v}\{\v^{\top}\nabla_t\v~|~\v\in\textrm{Im}(\X_t), \Vert{\v}\Vert=1\}$
\STATE $\lambda_t \gets \left({\v_{t,-}^{\top}\X_t^{\dagger}\v_{t,-}}\right)^{-1}$
\IF {$\lambda_t < 1$} 
\STATE $\X_{t+1}^{\textrm{drop}} \gets \frac{1}{1-\lambda_t}\left({\X_t - \lambda_t\v_{t,-}\v_{t,-}^{\top}}\right)$
\IF{$f(\X_{t+1}^{\textrm{drop}}) \leq f(\X_t)$} 
\STATE $\X_{t+1} \gets \X_{t+1}^{\textrm{drop}} $ \COMMENT{take a drop step}
\ENDIF
\ENDIF
\IF{drop step was not taken}
\STATE $\v_{t,+} \gets $ leading eigenvector of $-\nabla{}f(\X_t)$
\STATE $\eta^{\textrm{FW}}_{t} \gets \textrm{argmin}_{\eta\in[0,1]}f(\X_t + \eta(\v_{t,+}\v_{t,+}^{\top}-\X_t))$
\STATE $\X_{t+1}^{\textrm{FW}} \gets \X_t + \eta^{\textrm{FW}}_{t}(\v_{t,+}\v_{t,+}^{\top}-\X_t)$
\STATE $\eta^{\textrm{AFW}}_{t} \gets \textrm{argmin}_{\eta\in[0,\lambda_t]}f\left({\frac{1}{1-\eta}\left({\X_t -\eta\v_{t,-}\v_{t,-}^{\top}}\right)}\right)$
\STATE $\X_{t+1}^{\textrm{AFW}} \gets \frac{1}{1-\eta^{\textrm{AFW}}_{t}}\left({\X_t -\eta^{\textrm{AFW}}_{t}\v_{t,-}\v_{t,-}^{\top}}\right)$
\STATE $\u_{t,-} \gets$ uniformly dist. random vector in $\{\z\in\textrm{Im}(\X_t)~|~\Vert{\z}\Vert=1\}$
\STATE $\gamma_t \gets \left({\u_{t,-}^{\top}\X_t^{\dagger}\u_{t,-}}\right)^{-1}$
\STATE $\u_{t,+} \gets $ leading eigenvector of $\beta\gamma_t\u_{t,-}\u_{t,-}^{\top} - \nabla{}f(\X_t)$
%\STATE $\u_{t,+} \gets \textrm{argmin}_{\Vert{\u}\Vert=1}\u^{\top}\nabla{}f(\X_t)\u + \frac{\beta\gamma_t}{2}\Vert{\u\u^{\top}-\u_{t,-}\u_{t,-}^{\top}}\Vert_F^2$
%\STATE $\u_{t,+} \gets \in\arg\min_{\u}\{\u^{\top}\nabla_t\u + \frac{\beta\gamma_t}{2}\Vert{\u\u^{\top}-\u_{t,-}\u_{t,-}^{\top}}\Vert_F^2~|~ \Vert{\u}\Vert=1 \}$
\STATE $\X_{t+1}^{\textrm{PFW}} \gets \X_t + \gamma_t\left({\u_{t,+}\u_{t,+}^{\top} - \u_{t,-}\u_{t,-}^{\top}}\right)$
\STATE $\X_{t+1} \gets \textrm{argmin}_{\X\in\{\X_{t+1}^{\textrm{FW}}, \X_{t+1}^{\textrm{AFW}}, \X_{t+1}^{\textrm{PFW}}\}}f(\X)$  
\ENDIF
\ENDFOR
\end{algorithmic}
\end{algorithm}

\begin{lemma}[feasibility of Algorithm \ref{alg:AFW}]
The sequence $(\X_t)_{t\geq 1}$ produced by Algorithm \ref{alg:AFW} is always feasible w.r.t. $\mS^n$.
\end{lemma}
The correctness of this lemma could be easily verified using the following lemma, whose proof is given in the appendix.

\begin{lemma}\label{lem:stepsize}
Let $\X\in\mbS^n_+$ and $\v\in\textrm{Im}(\X), \v\neq \mathbf{0}$. Consider the matrix $\Y = \X - \lambda\v\v^{\top}$ for some $\lambda \geq 0$. If $\lambda \leq (\v^{\top}\X^{\dagger}\v)^{-1}$, then $\Y\succeq 0$. Moreover, if $\lambda = (\v^{\top}\X^{\dagger}\v)^{-1}$, then $\rank(\Y) = \rank(\X)-1$ and in particular  $\X^{\dagger}\v\in\textrm{Ker}(\Y)$.
%$\A-\lambda\v\v^{\top}\succeq 0$ if and only if $1-\lambda\v^{\top}\A^{\dagger}\v \geq 0$. 
\end{lemma}
%\begin{proof}
%Using the assumption that $\lambda \leq (\v^{\top}\X^{\dagger}\v)^{-1}$ we have that
%\begin{align*}
%1 - \lambda\v^{\top}\X^{\dagger}\v \geq 0 &\Longrightarrow 1- \lambda_1\left({\lambda{\X^{\dagger}}^{1/2}\v\v^{\top}{\X^{\dagger}}^{1/2}}\right) \geq 0\\
%&\Longrightarrow  \lambda_n\left({\I -  \lambda{\X^{\dagger}}^{1/2}\v\v^{\top}{\X^{\dagger}}^{1/2}}\right) \geq 0 \\
%&\Longrightarrow  \lambda_n\left({\X^{1/2}\left({\I -  \lambda{\X^{\dagger}}^{1/2}\v\v^{\top}{\X^{\dagger}}^{1/2}}\right)\X^{1/2}}\right) \geq 0 \\
%&\Longrightarrow  \X^{1/2}\left({\I -  \lambda{\X^{\dagger}}^{1/2}\v\v^{\top}{\X^{\dagger}}^{1/2}}\right)\X^{1/2} \succeq 0 \\
%&\underset{(a)}{\Longrightarrow}  \X -  \lambda\v\v^{\top}  \succeq 0, 
%\end{align*}
%where (a) relies on the assumption that $\v\in\textrm{Im}(\X)$ and so $\X^{1/2}{\X^{\dagger}}^{1/2}\v = \v$.

%Suppose now that $\lambda = (\v^{\top}\X^{\dagger}\v)^{-1}$. Note that since $\v\in\textrm{Im}(\X)$, we have that $\textrm{Im}(\Y) \subseteq \textrm{Im}(\X)$. Consider now the vector $\u = \X^{\dagger}\v\in\textrm{Im}(\X)$. It holds that
%\begin{align*}
%\Y\u = \left({\X-\frac{1}{\v^{\top}\X^{\dagger}\v}\v\v^{\top}}\right)\X^{\dagger}\v = \X\X^{\dagger}\v - \v = 0,
%\end{align*}
%implying that $\u\notin\textrm{Im}(\Y)$ and so, $\rank(\Y) < \rank(\X)$.
%\end{proof}

\subsection{Implementation details}\label{sec:implement}
We now discuss in more detail the efficient implementation of Algorithm \ref{alg:AFW}.
Let us fix  some iteration $t$ of the algorithm and denote by $\Pi_t\in\mbS^n_+$ the projection matrix onto the subspace $\textrm{Im}(\X_t)$. Note that $\Pi_t = \X_t\X_t^{\dagger}$. Let us assume for now that both matrices $\X_t^{\dagger}, \Pi_t$ are explicitly given (we shall discuss their efficient maintenance throughout the run of the algorithm in the sequel).

\paragraph*{Computing the vector $\v_{t,-}$:} This computation could be reduced to a standard leading eigenvector computation as follows. Computing $\v_{t,-}$, as defined in Algorithm \ref{alg:AFW}, corresponds to computing a vector in the set:
\begin{align}\label{eq:awayComp:1}
&\arg\max\nolimits_{\v\in\reals^n:\Pi_t\v=\v, \Vert{\v}\Vert =1}\v^{\top}\nabla{}f(\X_t)\v \nonumber \\
&= \arg\max\nolimits_{\v\in\reals^n:\Pi_t\v=\v, \Vert{\v}\Vert =1}\v^{\top}\Pi_t\nabla{}f(\X_t)\Pi_t\v \nonumber \\
&= \arg\max\nolimits_{\v\in\reals^n:\Pi_t\v=\v, \Vert{\v}\Vert =1}\v^{\top}\left({\Pi_t\nabla{}f(\X_t)\Pi_t + (1+\zeta)\Vert{\Pi_t\nabla{}f(\X_t)\Pi_t}\Vert\Pi_t}\right)\v,
\end{align}
where the last equality holds for any real scalar $\zeta$.

For any $\zeta >0$, the matrix  $\M=\Pi_t\nabla{}f(\X_t)\Pi_t + (1+\zeta)\Vert{\Pi_t\nabla{}f(\X_t)\Pi_t}\Vert\Pi_t$ in the RHS of \eqref{eq:awayComp:1} is positive semidefinite, all its eigenvectors with non-zero eigenvalue must lie in $\textrm{Im}(\Pi_t) = \textrm{Im}(\X_t)$, and $\lambda_1(\M) > 0$. Thus, when solving the maximization problem in the RHS of \eqref{eq:awayComp:1}, we can drop the explicit constraint $\Pi_t\v=\v$, and solve the standard leading eigenvector problem $\max_{\v\in\reals^n:\Vert{\v}\Vert=1}\v^{\top}\M\v$.  In particular, $\M$ need not be computed explicitly: standard fast iterative leading eigenvector methods (such as Lanczos-type algorithms) rely only on computing matrix-vector products w.r.t. the matrix $\M$, and thus, given $\Pi_t$ (or $\X_t^{\dagger}$), these could be implemented efficiently in $O(n^2)$ time per such matrix-vector product without explicitly computing $\M$.

\paragraph*{Computing the vector $\u_{t,-}$:}  $\u_{t,-}$ is a random vector distributed uniformly over the unit sphere in the space $\textrm{Im}(\X_t)$. Using  the rotation invariance of the multivariate standard Gaussian distribution $\mathcal{N}(\mathbf{0}, \I_n)$, we have that given a random Gaussian vector $\z\sim\mathcal{N}(\mathbf{0}, \I_n)$, the vector $\u_{t,-} = \frac{\Pi_t\z}{\Vert{\Pi_t\z}\Vert}$ is distributed uniformly  over the unit sphere in $\textrm{Im}(\X_t)$.

\paragraph*{Computing the vector $\u_{t,+}$:} As already explained above (see discussion following Eq. \eqref{eq:pws}), $\u_{t,+}$ can be computed using a standard leading eigenvector computation w.r.t. the matrix $\beta\gamma_t\u_{t,-}\u_{t,-}^{\top}-\nabla{}f(\X_t)$.\\

All computations discussed above require maintaining the projection matrix $\Pi_t$. Additionally, the computations of the scalars $\lambda_t, \gamma_t$ in Algorithm \ref{alg:AFW} require also access to the pseudoinverse matrix $\X_t^{\dagger}$. We now discuss how it is  possible to efficiently update either $\Pi_t$ or $\X_t^{\dagger}$ (or both) throughout the run of the algorithm.

\paragraph*{Maintaining only the pseudoinverse matrix $\X_t^{\dagger}$:} Similarly to the well-known  Sherman-Morrison-Woodbury formula for the fast update of  the matrix inverse after a rank-one update, \cite{meyer1973generalized} established a formula for updating the pseudoinverse matrix after a rank-one update. While the update in case of the pseudoinverse is substantially more involved, the computational complexity is similar. In particular,  on each iteration $t$ of Algorithm \ref{alg:AFW}, given $\X_{t}^{\dagger}$ and the quantities $\v_{t,-}$, $\lambda_t$, $\v_{t,+}$, $\u_{t,-}$, $\u_{t,+}$, $\eta_t^{\textrm{FW}}$, $\eta_t^{\textrm{AFW}}$, $\eta_t^{\textrm{PFW}}$, $\gamma_t$, the pseudoinverse matrix of the next iterate, $\X_{t+1}^{\dagger}$, can be computed in $O(n^2)$ time \footnote{while \cite{meyer1973generalized} considers the general case of complex non-square matrices which leads to 6 different types of updates for the pseudoinverse, here since we consider symmetric matrices, only 3 of the 6 possibilities need to be considered}. For more concrete details see Lemma \ref{lem:pinvUpdate} in the appendix. Since we have that $\Pi_{t+1} = \X_{t+1}\X_{t+1}^{\dagger}$, it is not required to also maintain $\Pi_{t+1}$ explicitly. This leads to the following theorem.
\begin{theorem}
Algorithm \ref{alg:AFW} admits an implementation such that on each iteration $t$, aside from the computation of the eigenvectors $\v_{t,-}, \v_{t,+}, \u_{t,+}$, all other computations can be carried out in $O(n^2)$ time.
\end{theorem}

\paragraph*{Maintaining only the projection matrix $\Pi_t$:} A different possibility is not to maintain the pseudoinverse $\X_t^{\dagger}$ explicitly, but only the projection matrix onto $\textrm{Im}(\X_t)$, $\Pi_t$\footnote{it may be helpful to recall that for any orthonormal basis for $\textrm{Im}(\X_t)$, $\{\q_1,\dots,\q_{\rank(\X_t)}\}$, we have that $\Pi_t = \sum_{i=1}^{\rank(\X_t)}\q_i\q_i^{\top}$}. Let us consider the four possible cases.
\textbf{I.} In case a drop step is performed on some iteration $t$, then according to Lemma \ref{lem:stepsize} we have that $\z = \X_t^{\dagger}\v_{t,-}$ satisfies $\z\in\textrm{Ker}(\X_{t+1})$. Note also that $\z\in\textrm{Im}(\X_t)$. This implies that $\Pi_{t+1} = (\I-\frac{1}{\Vert{\z}\Vert^2}\z\z^{\top})\Pi_t = \Pi_t - \frac{1}{\Vert{\z}\Vert^2}\z\z^{\top}$. The vector $\z$ could be approximated to arbitrary accuracy by solving the simple least squares problem: $\min_{\y\in\reals^n}\Vert{\X_t\y-\v_{t,-}}\Vert^2$ \footnote{Recall each solution to this LS problem is of the form $\X_t^{\dagger}\v_{t,-} + \q$, for some $\q\in\textrm{Ker}(\X_t)$}.  \textbf{II.} In case a standard Frank-Wolfe step is performed and $\eta_t^{\textrm{FW}} < 1$, defining $\z := \v_{t,+} - \Pi_t\v_{t,+}$, we have that $\Pi_{t+1} = \Pi_t + \frac{1}{\Vert{\z}\Vert^2}\z\z^{\top}$ (if $\z = \mathbf{0}$, we simply set $\Pi_{t+1} = \Pi_t$). In case $\eta_t^{\textrm{FW}} = 1$, we simply have that $\Pi_{t+1} = \v_{t,+}\v_{t,+}^{\top}$. \textbf{III.} In case an away step is taken (which is not a drop step), we have that  $\textrm{Im}(\X_{t+1}) = \textrm{Im}(\X_t)$, and so $\Pi_{t+1} = \Pi_t$. \textbf{IV.} If a pairwise step is taken, then we update $\Pi_t$ using two steps: a drop step (as in Case \textbf{I} above) w.r.t. the direction $\u_{t,-}$, and then adding the direction  $\u_{t,+}$ (as in Case \textbf{II} above). That is, denoting $\w := \X_t^{\dagger}\u_{t,-}$, $\Pi_{t+1/2}:= \Pi_t - \frac{1}{\Vert{\w}\Vert^2}\w\w^{\top}$, and then denoting $\z := \u_{t,+} - \Pi_{t+1/2}\u_{t,+}$, we have that $\Pi_{t+1} = \Pi_{t+1/2} + \frac{1}{\Vert{\z}\Vert^2}\z\z^{\top}$ (or simply $\Pi_{t+1} = \Pi_{t+1/2}$ in case $\z=\mathbf{0}$). Finally, as with the computation of the product $\X_t^{\dagger}\v_{t,-}$ in Case \textbf{I}, computing the scalars $\lambda_t,\gamma_t$, which requires computing a matrix-vector product with the pseudoinverse matrix $\X_t^{\dagger}$, can be done by solving a simple least squares problem.

\paragraph*{Line-search computations:} There are two such computations in the algorithm, one for computing $\eta_t^{\textrm{FW}}$ --- the step-size for a standard FW update, and the second for computing $\eta_t^{\textrm{AFW}}$ --- the step-size for the away step. First, let us see that both of these take the same form. We can rewrite the updates for $\eta_t^{\textrm{AFW}}$ and $\X_{t+1}^{\textrm{AFW}}$ in Algorithm \ref{alg:AFW} as follows:
\begin{align*}
\eta^{\textrm{AFW}}_{t} &\gets \textrm{argmin}_{\eta\in[0,\lambda_t]}f\left({\X_t + \frac{\eta}{1-\eta}\left({\X_t -\v_{t,-}\v_{t,-}^{\top}}\right)}\right),\\
\X_{t+1}^{\textrm{AFW}} &\gets \X_t + \frac{\eta^{\textrm{AFW}}_{t}}{1-\eta^{\textrm{AFW}}_{t}}\left({\X_t -\v_{t,-}\v_{t,-}^{\top}}\right).
\end{align*}
Thus, we can rewrite these updates to have the same form as those for the standard FW update:
\begin{align*}
\tilde{\eta}^{\textrm{AFW}}_{t} &\gets \textrm{argmin}_{\tilde{\eta}\in[0,\frac{\lambda_t}{1-\lambda_t}]}f\left({\X_t + \tilde{\eta}\left({\X_t -\v_{t,-}\v_{t,-}^{\top}}\right)}\right),\\
\X_{t+1}^{\textrm{AFW}} &\gets \X_t + \tilde{\eta}^{\textrm{AFW}}_{t}\left({\X_t -\v_{t,-}\v_{t,-}^{\top}}\right).
\end{align*}
Such line-search computations are standard in FW algorithms \cite{jaggi2013revisiting, lacoste2015global}. In particular, when $f$ is quadratic these admit a simple closed-form solution. In other cases, they amount to the minimization of a 1-dimensional convex function over a closed interval which can be done efficiently via bisection. Moreover, since our algorithm already requires the knowledge of the smoothness constant $\beta$, both of these computations could be replaced with performing a line search w.r.t. the local quadratic upper-bound of $f$:
\begin{align*}
\eta^{\textrm{FW}} &\gets \textrm{argmin}_{\eta\in[0,1]}f(\X_t) + \eta\langle{\v_{t,+}\v_{t,+}^{\top}-\X_t,\nabla{}f(\X_t)}\rangle + \frac{\eta^2\beta}{2}\Vert{\v_{t,+}\v_{t,+}^{\top}-\X_t}\Vert_F^2,\\
\eta^{\textrm{AFW}} &\gets \textrm{argmin}_{\eta\in[0,\frac{\lambda_t}{1-\lambda_t}]}f(\X_t) + \eta\langle{\X_t -\v_{t,-}\v_{t,-}^{\top},\nabla{}f(\X_t)}\rangle + \frac{\eta^2\beta}{2}\Vert{\X_t - \v_{t,-}\v_{t,-}^{\top}}\Vert_F^2.
\end{align*} 
These are quadratic minimization problems and thus admit simple closed-form solutions. In particular, it is a simple observation that using these  computations, instead of those in Algorithm \ref{alg:AFW}, will not change our analysis and all theoretical derivations will remain intact.

\section{Convergence Rate Analysis}\label{sec:analysis}
We turn to formally state and prove our main result --- the linear convergence rate of Algorithm \ref{alg:AFW}. Throughout this section we let $h_t := f(\X_t) - f^*$ denote the approximation error of Algorithm \ref{alg:AFW} on any iteration $t\geq 1$.
\begin{theorem}[convergence of Algorithm \ref{alg:AFW}]\label{thm:main}
The sequence of iterates $(\X_t)_{t\geq 1}$ produced by Algorithm \ref{alg:AFW} satisfies:
\begin{alignat}{2}
&\forall t\geq 1: \qquad  &&h_{t+1}  \leq  h_t. \label{eq:mainthm:res:0}\\
&\forall t\geq \rank(\X_1)+2: \qquad &&h_t  \leq  \frac{8\beta}{t-\rank(\X_1)+4}. \label{eq:mainthm:res:1}
%\forall t\geq 2: \qquad h_t  \leq \min\left\{h_{t-1}, \frac{8\beta}{t-r_1+4}\right\}.
\end{alignat}
Moreover, there exist scalars $\mathcal{E} =O\left({\min\left\{{\frac{\delta^2}{\beta}, \delta, \alpha\lambda_{r^*}^2, \frac{\delta^3}{\beta^2}}\right\} }\right)$, $\mathcal{E}_1 = O\left({\frac{\delta^2}{\beta}}\right)$, $\mathcal{E}_2 = O(\alpha\lambda_{r^*}^2)$ such that on each iteration $t$ in which a \underline{drop step is not taken}:  \\
\noindent{}If $r^* =1$ and $h_t < \mathcal{E}_1$ then
\begin{align}\label{eq:mainthm:res:2}
h_{t+1} \leq h_t\left({1-\min\{\frac{\delta}{12\beta}, \frac{1}{2}\}}\right).
\end{align}
If $r^* =n$ and $h_t < \mathcal{E}_2$ then
\begin{align}\label{eq:mainthm:res:4}
h_{t+1} \leq h_t\left({1-\frac{\alpha}{512\beta{}n}}\right).
\end{align}
If  $n > r^* \geq 2$ and $h_t <  \mathcal{E}$ then, 
\begin{align}\label{eq:mainthm:res:3}
\E[h_{t+1}|\X_t] \leq h_t\left({1-\frac{1}{r^*}\min\left\{{\frac{\delta}{48G}\min\{\frac{\delta}{3\beta},\lambda_{r^*}\}
,~ \frac{\alpha}{512\beta}}\right\}}\right),
%\E_{\u_{t,-}}[h_{t+1}~|~\X_t] \leq h_t\left({1-\frac{1}{12r^*}\min\left\{{\frac{\delta}{G}\min\{\frac{\delta}{9\beta},\lambda_r(\X^*)\}
%,~ \frac{\alpha}{12\beta}}\right\}}\right),
\end{align}
where the expectation is w.r.t. the random choice of $\u_{t,-}$, and $G := \sup_{\X\in\mS^n}\Vert{\nabla{}f(\X)}\Vert$.

Finally, up to any iteration $t$, the overall number of drop steps cannot exceed $(t+\rank(\X_1)-2)/2$.
\end{theorem}

Before proving the theorem let us make a few comments. Result \eqref{eq:mainthm:res:0} follows in a straightforward manner from the design of the algorithm. Results \eqref{eq:mainthm:res:1}, \eqref{eq:mainthm:res:2} follow essentially from already known analyses of the Frank-Wolfe method \cite{jaggi2013revisiting} and \cite{garber2023linear}. The main novel result is the (expected) linear convergence rate in the case $n > r^* \geq 2$ given in Eq. \eqref{eq:mainthm:res:3}. Note that the rate in  \eqref{eq:mainthm:res:3} depends explicitly on the rank of optimal solutions $r^*$, which is well-known to be unavoidable in worst case, see for instance \cite{jaggi2011convex}. Also note that the $\frac{\alpha}{\beta}$ term inside the min in the RHS of \eqref{eq:mainthm:res:3} corresponds (up to a universal constant) to the standard linear convergence rate of (unaccelerated) gradient methods, such as the block Frank-Wolfe method \cite{allen2017linear}.  Finally, Result \eqref{eq:mainthm:res:4} for the case $r^* =n$, i.e., the optimal solutions lie in the relative interior of $\mS^n$, follows as a simplified case of  Result  \eqref{eq:mainthm:res:3}.

\subsection{Proof of Theorem \ref{thm:main}}

%The proofs of Results \eqref{eq:mainthm:res:0}, \eqref{eq:mainthm:res:1} and \eqref{eq:mainthm:res:2} follow essentially from the simple observations that I. the approximation error of Algorithm \ref{alg:AFW} never increases,  and II. that each iteration which is not a drop step  reduces the objective function value by at least the amount that a standard Frank-Wolfe step with line-search will reduce. Then, \eqref{eq:mainthm:res:1}  follows from  classical analysis of Frank-Wolfe, e.g., \cite{jaggi2013revisiting}, and \eqref{eq:mainthm:res:2} follows from the analysis for the case $r^*=1$ given in \cite{garber2023linear}. Thus, the majority of the following analysis is devoted to the proof of Result \eqref{eq:mainthm:res:3} which is the main novelty. 

\begin{center}
%\begin{table}
\begin{tabular}{ |p{1.4cm}|p{10.7cm}|}
%\begin{tabular}{ |l|l| }
  \hline
  \multicolumn{2}{|c|}{Additional notation for the proof of Theorem \ref{thm:main}} \\
  \hline
  $\nabla_t$ & $\nabla{}f(\X_t)$ \\
  %$h_t$ & $f(\X_t) - f^*$ \\
  $\delta_t$ & $\lambda_{n-r^*}(\nabla{}f(\X_t)) - \lambda_{n-r^*+1}(\nabla{}f(\X_t))$ (when $r^* < n$)\\
  $\delta_{t,r^*}$ & $\lambda_{n-r^*+1}(\nabla{}f(\X_t)) - \lambda_{n}(\nabla{}f(\X_t))$ \\  
  $r_t$ & $\rank(\X_t)$\\
  $\P_t\in\mbS^n_{+}$ & projection matrix onto span of $r^*$ eigenvectors of $\nabla_t$ corresponding to smallest eigenvalues of $\nabla_t$ \\
  $\P_t^{\perp}\in\mbS^n_{+}$ & projection matrix onto span of $n-r^*$ eigenvectors of $\nabla_t$ corresponding to largest eigenvalues of $\nabla_t$. If $r^*=n$ set $\P_t^{\perp} = \mathbf{0}_{n\times n}$\\
  $\E_t[\cdot]$ & $\E_{\u_{t,-}}[\cdot|\X_t]$ --- the expectation w.r.t. $\u_{t,-}$ conditioned on $\X_t$ \\
  \hline
\end{tabular}
%\caption{Notation for the proof of Theorem \ref{thm:main}}
%\end{table}
\end{center}

\vspace{8pt}

\paragraph*{Roadmap for the proof:}  As mentioned above, the most challenging part is the proof of Result  \eqref{eq:mainthm:res:3}.  We consider the run of Algorithm \ref{alg:AFW} in two phases. The first one is a finite burn-in phase, before the strict complementarity condition has been ``detected'', and the standard FW steps are used to drive a provable sublinear convergence rate (Result \eqref{eq:mainthm:res:1}). Once the approximation error $h_t$ reaches a critical value (which naturally scales with the strict complementarity measure $\delta$), we enter the second phase of the run in which the convergence is linear, in expectation. In this second phase, the algorithm automatically adapts to the optimal rank $r^*$ --- on each iteration $t$ in which a drop step is not taken, it must hold that $\rank(\X_t) \leq r^*$ (Lemma \ref{lem:dropstep}). On each such iteration we then consider two cases. In the first case, the current iterate $\X_t$ is not sufficiently aligned with the principal subspace of the gradient $\nabla{}f(\X_t)$, in the sense that $\trace(\P_t^{\perp}\X_t) > c{}h_t$ for some constant $c$. In this case, Lemmas \ref{lem:pfwstep} and \ref{lem:expectedProj} guarantee that a pairwise step will reduce the approximation error, in expectation, by $\Omega(h_t)$ (compatible with a linear convergence rate, in expectation). Here it is important that $\rank(\X_t) \leq r^*$ which, together with the quadratic growth property, guarantees the pairwise step can be taken with sufficiently large step-size. 
In the other case, the iterate $\X_t$ is sufficiently aligned with the principal subspace of $\nabla{}f(\X_t)$, and we can use the quadratic growth of $f$ to establish a Polyak-\L{}ojasiewicz-type (PL) inequality (a well-known property used to derive linear convergence rates for first-order methods) on the face of $\mS^n$ corresponding to this subspace (Lemma \ref{lem:eigengap}). With this inequality in place, we can establish that either  a standard FW step or an away step will lead to a $\Omega(h_t)$ (deterministic) reduction in approximation error, which is again compatible with a linear convergence rate.

%The proof of Theorem \ref{thm:main} is  constructed from several lemmas that analyze the benefit of each one of the steps employed by Algorithm \ref{alg:AFW}. It is then established how their combination yields the theorem. 

The following lemma establishes conditions under which a drop step will be taken.
\begin{lemma}[drop step]\label{lem:dropstep}
Fix an iteration $t$ of Algorithm \ref{alg:AFW}  for which  $h_t < \min\{\frac{\delta_t^2}{18\beta}, \frac{\delta_t}{6}\}$. If $\rank(\X_t) > r^*$, then  $f(\X^{\textrm{drop}}_{t+1}) \leq f(\X_t)$.
\end{lemma}
\begin{proof}
Recall that $\X_{t+1}^{\textrm{drop}} = \frac{1}{1-c}(\X_t-c\v\v^{\top})$ for  $c = (\v^{\top}\X_t^{\dagger}\v)^{-1}$ and  that $\v\in\textrm{argmax}\{\v^{\top}\nabla_t\v~|~\v\in\textrm{Im}(\X_t), \Vert{\v}\Vert = 1\}$. Using the smoothness of $f$ we have that,
\begin{align}\label{eq:lem:dropstep:1}
f(\X_{t+1}^{\textrm{drop}}) &\leq f(\X_t) + \langle{\frac{1}{1-c}(\X_t-c\v\v^{\top}) - \X_t, \nabla_t}\rangle + \frac{\beta}{2}\Vert{\frac{1}{1-c}(\X_t-c\v\v^{\top}) - \X_t}\Vert_F^2 \nonumber \\
&\leq f(\X_t)  + \frac{c}{1-c}\langle{\X_t-\v\v^{\top}, \nabla_t}\rangle + \frac{\beta{}c^2}{(1-c)^2}.
\end{align}
Since $\dim\textrm{Im}(\X_t) \geq r^*+1$, it follows from the Poincar\'{e} inequality (e.g., Theorem 4.4 in \cite{calafiore2014optimization}) that $\v^{\top}\nabla_t\v \geq \lambda_{n-r^*}(\nabla_t) = \lambda_{n-r^*+1}(\nabla_t)+\delta_t$\footnote{Poincar\'{e} inequality states that for any $\A\in\mbS^n$ and a subspace $\mV\subseteq\reals^n$ of dimension $k$, there must exist a unit vector $\w\in\mV$ such that $\w^{\top}\A\w \geq \lambda_{n-k+1}(\A)$.\label{footnote:pi}}. On the other hand, we argue that $\langle{\X_t,\nabla_t}\rangle \leq  \lambda_{n-r^*+1}(\nabla_t)+\delta_t/2$. To see why the latter is true, let us assume by way of contradiction that it is not, and  let us denote the matrix $\Y_{\eta} = (1-\eta)\X_t + \eta\v_{t,+}\v_{t,+}^{\top}$ for some $\eta\in[0,1]$ to be specified later, and $\v_{t,+}$ as defined in Algorithm \ref{alg:AFW}. Using the smoothness of $f$ again we have that,
\begin{align*}
f(\Y_{\eta}) - f^* &\leq f(\X_t) + \eta\langle{\v_{t,+}\v_{t,+}^{\top}-\X_t,\nabla_t}\rangle + \frac{\beta\eta^2}{2}\Vert{\v_{t,+}\v_{t,+}^{\top}-\X_t}\Vert_F^2 - f^* \\
&\leq f(\X_t) - f^* + \eta\left({\lambda_n(\nabla_t) -  \lambda_{n-r^*+1}(\nabla_t) -\delta_t/2}\right) + \eta^2\beta \\
&\leq f(\X_t) - f^* - \eta\left({\frac{\delta_t}{2} - \eta\beta}\right),
\end{align*}
where the middle inequality is due to the definition of $\v_{t,+}$ in the algorithm and our assumption by way of contradiction that $\langle{\X_t,\nabla_t}\rangle >  \lambda_{n-r^*+1}(\nabla_t)+\delta_t/2$.

Setting $\eta = \min\{\frac{\delta_t}{3\beta}, 1\}$ yields then
\begin{align*}
f(\Y_{\eta}) - f^* \leq f(\X_t) - f^* - \min\{\frac{\delta_t}{3\beta}, 1\}\frac{\delta_t}{6}.
\end{align*}
Now we see that  if $h_t < \min\{\frac{\delta_t^2}{18\beta}, \frac{\delta_t}{6}\}$, as the lemma assumes,  we indeed get a contradiction.
%Now, if $\frac{\delta}{3\beta} \leq 1$, setting $\gamma = \frac{\delta}{3\beta}$ and recalling the assumption of the lemma that $h_t < \frac{\delta_t^2}{18\beta}$ yields that $f(\Y_{\gamma}) < f^*$, which is impossible. Otherwise, if $\frac{\delta}{3\beta} > 1$, setting $\gamma =1$ and recalling the assumption that $h_t < \delta_t/6$, again leads to a contradiction. 
%Thus, we can conclude that indeed $\langle{\X_t,\nabla_t}\rangle \leq \lambda_{n-r^*+1}+\delta_t/2$.

Going back to \eqref{eq:lem:dropstep:1}, we can now write
\begin{align}\label{eq:lem:dropstep:2}
f(\X_{t+1}^{\textrm{drop}}) -f^* \leq f(\X_t) - f^* - \frac{\delta_tc}{2(1-c)} + \frac{\beta{}c^2}{(1-c)^2}.
\end{align}
Observe that if $\frac{c}{1-c} \leq \frac{\delta_t}{2\beta}$, then indeed $f(\X_{t+1}^{\textrm{drop}}) \leq f(\X_t)$. Note furthermore that it must be the case that $\frac{c}{1-c} \leq \frac{\delta_t}{3\beta}$: the function $g(c)=\frac{c}{1-c}$ is monotone increasing on the interval $[0,1)$, and thus we can in principle replace the scalar $c$ with some $c' \leq c$ such that $\frac{c'}{1-c'} < \frac{\delta_t}{3\beta}$. Denoting the resulting matrix by $\X'_{t+1} = \frac{1}{1-c'}(\X_t - c'\v\v^{\top})$ and replacing $c$ with $c'$ in \eqref{eq:lem:dropstep:2} will then yield, 
\begin{align*}
f(\X'_{t+1}) -f^* \leq f(\X_t) - f^* - \frac{\delta_t^2}{18\beta}.
\end{align*}
However, using the assumption of the lemma that $h_t < \frac{\delta_t^2}{18\beta}$, this would lead to a contradiction. 
Thus, we can conclude that indeed $f(\X_{t+1}^{\textrm{drop}}) \leq f(\X_t)$.
\end{proof}

%\begin{lemma}[Pairwise step I]
%Let $\v\in\textrm{Im}(\X_t)$ and consider the unit vector $\v'$ such that
%\begin{align*}
%\v'  \in\arg\min_{\u}\{\u^{\top}\nabla_t\u + \frac{\beta\lambda_t}{2}\Vert{\u\u^{\top}-\v\v^{\top}}\Vert_F^2~|~ \Vert{\u}\Vert=1\}.
%\end{align*} 
%and consider the matrix $\Y_{t+1} = \X_{t+1} + \eta_t(\v'\v'^{\top} - \v\v^{\top})$. Then,
%\begin{align*}
%f(\Y_{t+1}) \leq f(\X_t) - \eta_t\Vert{\P_t^{\perp}\v}\Vert^2\left({\delta_t-\eta_t\beta}\right).
%\end{align*}
%\end{lemma}
%\begin{proof}
%...
%\end{proof}

Since drop steps, as analyzed in the previous lemma, do not necessarily lower the objective value by a bounded amount, we have the following simple observation that upper-bounds the number of such steps possible until any iteration $t\geq 1$.

\begin{observation}\label{obsrv:dropsteps}
The overall number of drop steps up to the beginning of any iteration $t$ of Algorithm \ref{alg:AFW} cannot be larger than $(t + \rank(\X_1) - 2)/2$.\end{observation}
\begin{proof}
According to Lemma \ref{lem:stepsize}, each drop step reduces the rank of the resulting matrix. On the other hand, all other steps increase the rank by at most one.  Thus, denoting the number of drop steps up to the beginning of some iteration $t$ by $t_{\textrm{drop}}$, it must hold that
\begin{align*}
1 \leq \rank(\X_t) \leq \rank(\X_1) - t_{\textrm{drop}} + (t-1 -  t_{\textrm{drop}}),
\end{align*}
which yields the desired bound.

%Thus, up to any iteration $t$ (not including the update from $\X_t$ to $\X_{t+1}$), the number of drop steps, denoted by $t_{\textrm{drop}}$, must  satisfy
%\begin{align}
%t_{\textrm{drop}} \leq \frac{r_1+t-2}{2}.
%\end{align}  
\end{proof}

The following lemma (deterministically) bounds the reduction in function value due to a pairwise step. It is complemented by Lemma \ref{lem:expectedProj} which further lower-bounds the expectation of the random variable  $\Vert{\P_t^{\perp}\u_{t,-}}\Vert^2$.
\begin{lemma}[pairwise step]\label{lem:pfwstep}
Fix some iteration $t$ of Algorithm \ref{alg:AFW} for which it holds that $h_t \leq  \min\{\frac{\delta_t^3}{24\beta^2}, \frac{\delta_t^2}{12\beta}\}$. Then,
\begin{align*}
f(\X^{\textrm{PFW}}_{t+1})  \leq f(\X_t) - \min\{\frac{\delta_t^2}{9\beta},\frac{\gamma_t\delta_t}{3}\}\Vert{\P_t^{\perp}\u_{t,-}}\Vert^2.
%f(\X_{t+1}) -f^* \leq f(\X_t)-f^* - \eta_t\Vert{\P_t^{\perp}\v_{t,-}}\Vert^2\left({\delta_t-\eta_t\beta}\right).
\end{align*}
\end{lemma}
\begin{proof}
Write $\u_{t,-}$ as $\u_{t,-} = \tau\z + \sqrt{1-\tau^2}\w$, where $\z=\frac{\P_t^{\perp}\u_{t,-}}{\Vert{\P_t^{\perp}\u_{t,-}}\Vert}$, $\tau = \Vert{\P_t^{\perp}\u_{t,-}}\Vert$,  and $\w=\frac{\P_t\u_{t,-}}{\Vert{\P_t\u_{t,-}}\Vert}$. Note that by definition of $\P_t$ and $\P_t^{\perp}$ we have that 
\begin{align}\label{eq:lem:pfw:0}
\z^{\top}\nabla_t\z - \w^{\top}\nabla_t\w \geq \lambda_{n-r^*}(\nabla_t) - \lambda_{n-r^*+1}(\nabla_t) = \delta_t. 
\end{align}
Consider now some unit vector $\u = \tau'\z + \sqrt{1-\tau'^2}\w$ with $0\leq \tau'\leq\tau$. Fix some step-size $\eta\in[0,1]$ and denote $\Y_{\eta} = \X_t + \eta(\u\u^{\top}-\u_{t,-}\u_{t,-}^{\top})$. Using the smoothness of $f$ it holds that,
\begin{align}\label{eq:lem:pfw:1}
%f(\Y) &\leq f(\X_t) + \eta\langle{\u_{t,+}\u_{t,+}^{\top}-\u_{t,-}\u_{t,-}^{\top}, \nabla_t}\rangle + \frac{\eta^2\beta}{2}\Vert{{\u_{t,+}\u_{t,+}^{\top}-\u_{t,-}\u_{t,-}^{\top}}\Vert}_F^2 \nonumber \\
 f(\Y_{\eta}) &\leq f(\X_t) + \eta\langle{\u\u^{\top}-\u_{t,-}\u_{t,-}^{\top}, \nabla_t}\rangle + \frac{\eta^2\beta}{2}\Vert{{\u\u^{\top}-\u_{t,-}\u_{t,-}^{\top}}\Vert}_F^2 \nonumber \\
 &= f(\X_t) + \eta\left({(\tau'^2-\tau^2)\z^{\top}\nabla_t\z -(\tau'^2-\tau^2)\w^{\top}\nabla_t\w}\right) + \eta^2\beta\left({1-(\u^{\top}\u_{t,-})^2}\right) \nonumber \\
 &\underset{(a)}{\leq} f(\X_t) - \eta(\tau^2-\tau'^2)\delta_t +\eta^2\beta\left({1 - (\tau\tau' + \sqrt{1-\tau^2}\sqrt{1-\tau'^2})^2}\right) \nonumber\\
 &= f(\X_t) -\eta(\tau^2-\tau'^2)\delta_t \nonumber \\
 &~~~+ \eta^2\beta\left({1 - \left({\tau^2\tau'^2 + (1 - \tau^2-\tau'^2 + \tau^2\tau'^2) + 2\tau\tau'\sqrt{1-\tau^2}\sqrt{1-\tau'^2}}\right)}\right) \nonumber \\
  &= f(\X_t) -\eta(\tau^2-\tau'^2)\delta_t + \eta^2\beta\left({\tau^2 + \tau'^2  - 2\tau^2\tau'^2 - 2\tau\tau'\sqrt{1-\tau^2}\sqrt{1-\tau'^2}}\right) \nonumber\\
  &\underset{(b)}{\leq}    f(\X_t) -\eta(\tau^2-\tau'^2)\delta_t + \eta^2\beta\left({\tau^2 + \tau'^2  - 2\tau^2\tau'^2 - 2\tau\tau'(1-\tau^2)}\right) \nonumber\\
    &=    f(\X_t) -\eta(\tau^2-\tau'^2)\delta_t + \eta^2\beta(\tau-\tau')^2  + 2\eta^2\beta\tau^2\tau'(\tau-\tau') \nonumber\\
    &= f(\X_t) - \eta(\tau-\tau')\left({(\tau+\tau')\delta_t - \eta\beta\left({(\tau-\tau')+2\tau^2\tau'}\right)}\right),
\end{align}
where (a) is due to Eq. \eqref{eq:lem:pfw:0} and recalling that $\tau'\in[0,\tau]$, and (b) also follows since $\tau' \leq \tau$.

First note that setting $\tau' =0$ in \eqref{eq:lem:pfw:1} yields
\begin{align*}
f(\Y_{\eta}) \leq f(\X_t) - \eta\tau^2\left({\delta_t - \eta\beta}\right).
\end{align*}
Using the step-size $\eta^* =  \min\{\frac{\delta_t}{2\beta}, \gamma_t\}$, which satisfies that $\Y_{\eta^*}\in\mS^n$ (see Lemma \ref{lem:stepsize}), then gives
\begin{align*}
f(\Y_{\eta^*}) - f^* \leq  h_t - \frac{ \min\{\frac{\delta_t}{2\beta}, \gamma_t\}\delta_t\tau^2}{2},
\end{align*}
which implies that
\begin{align}\label{eq:lem:pfw:2}
\tau^2 \leq \frac{2h_t}{ \min\{\frac{\delta_t}{2\beta}, \gamma_t\}\delta_t}.
\end{align}

Consider now the case $\tau' = (1-c)\tau$ for some scalar $c\in[0,1]$ to be specified in the sequel. Using the smoothness of $f$ again, the definition of $\u_{t,+}$ in Algorithm \ref{alg:AFW} (see also Eq. \eqref{eq:pws}), and applying \eqref{eq:lem:pfw:1} with $\eta=\gamma_t$ (note that by definition $\gamma_t\in(0,1]$) we have that,
\begin{align*}
f(\X^{\textrm{PFW}}_{t+1}) &\leq f(\X_t) + \gamma_t\langle{\u_{t,+}\u_{t,+}^{\top}-\u_{t,-}\u_{t,-}^{\top}, \nabla_t}\rangle + \frac{\gamma_t^2\beta}{2}\Vert{{\u_{t,+}\u_{t,+}^{\top}-\u_{t,-}\u_{t,-}^{\top}}\Vert}_F^2 \nonumber \\
&\leq f(\X_t) - \gamma_tc\tau^2\left({\delta_t - \gamma_t\beta{}(c + 2\tau^2)}\right).
\end{align*}

Let us now set $c = \min\{\frac{\delta_t}{3\gamma_t\beta},1\}$ to obtain
\begin{align*}
f(\X^{\textrm{PFW}}_{t+1}) &\leq f(\X_t) - \gamma_t\min\{\frac{\delta_t}{3\gamma_t\beta},1\}\tau^2\left({\frac{2\delta_t}{3} - 2\gamma_t\beta\tau^2}\right) \\
& \leq f(\X_t) - \gamma_t\min\{\frac{\delta_t}{3\gamma_t\beta},1\}\tau^2\left({\frac{2\delta_t}{3} - \frac{4\gamma_t\beta{}h_t}{\min\{\frac{\delta_t}{2\beta}, \gamma_t\}\delta_t}}\right),
\end{align*}
where the last inequality follows from plugging \eqref{eq:lem:pfw:2} into the term $2\gamma_t\beta\tau^2$.

Thus, recalling that $\tau^2 = \Vert{\P_t^{\perp}\u_{t,-}}\Vert^2$ and that $\gamma_t\in(0,1]$, we have that if $h_t \leq  \frac{\delta_t^2\min\{\frac{\delta_t}{2\beta},1\}}{12\beta} \leq \frac{\delta_t^2\min\{\frac{\delta_t}{2\beta}, \gamma_t\}}{12\gamma_t\beta}$, then we indeed have that
\begin{align*}
f(\X^{\textrm{PFW}}_{t+1}) \leq f(\X_t) - \frac{\gamma_t\delta_t}{3}\min\{\frac{\delta_t}{3\gamma_t\beta},1\}\Vert{\P_t^{\perp}\u_{t,-}}\Vert^2.
\end{align*}

%Suppose $\Vert{\P_t^{\perp}\v_{t,-}}\Vert >0$, otherwise the lemma holds trivially. Denote $\w_t = \frac{\P_t\v_{t,-}}{\Vert{\P_t\v_{t,-}}\Vert}$. Using the choice of $\v_{t,+}$ and the smoothness of $f$ we have that,
%\begin{align*}
%f(\X_{t+1}) &\leq f(\X_t) + \eta_t\langle{\v_{t,+}\v_{t,+}^{\top}-\v_{t,-}\v_{t,-}^{\top}, \nabla_t}\rangle + \frac{\eta_t\beta}{2}\Vert{{\v_{t,+}\v_{t,+}^{\top}-\v_{t,-}\v_{t,-}^{\top}}\Vert}_F^2 \\
% &\leq f(\X_t) + \eta_t\langle{\w_t\w_t^{\top}-\v_{t,-}\v_{t,-}^{\top}, \nabla_t}\rangle + \frac{\eta_t\beta}{2}\Vert{{\w_t\w_t^{\top}-\v_{t,-}\v_{t,-}^{\top}}\Vert}_F^2 \\
%&\leq f(\X_t) + \left({\w_t^{\top}\nabla_t\w_t - \left({(1-\Vert{\P_t\v_{t,-}}\Vert^2)\lambda_{n-r^*}(\nabla_t) + \Vert{\P_t\v_{t,-}}\Vert^2\w_t^{\top}\nabla_t\w_t}\right)}\right) \\
%&~~~+ \eta_t\beta\left({1 - \Vert{\P_t\v_{t,-}}\Vert^2}\right) \\
%&=f(\X_t) + \eta_t\Vert{\P_t^{\perp}\v_{t,-}}\Vert^2\left({-\delta_t + \eta_t\beta}\right).
%\end{align*}
\end{proof}

\begin{lemma}\label{lem:expectedProj}
Fix some iteration $t$ of Algorithm \ref{alg:AFW}. It holds that $\E_t[\Vert{\P_t^{\perp}\u_{t,-}}\Vert^2] \geq \frac{1}{r_t}\trace(\P_t^{\perp}\X_t)$.
\end{lemma}
\begin{proof}
Write the eigen-decomposition of $\X_t$ as $\X_t = \U_t\Lambda_t\U_t^{\top}$, where $\Lambda_t$ is $r_t\times r_t$ and $\U_t$ is $n\times r_t$. It holds that,
\begin{align*}
\trace(\P_t^{\perp}\X_t) &= \trace(\P_t^{\perp}\U_t\Lambda_t\U_t^{\top}) = \trace(\U_t^{\top}\P_t^{\perp}\U_t\Lambda_t) 
\leq \lambda_1(\X_t)\trace(\U_t^{\top}\P_t^{\perp}\U_t) \\
&\leq \trace(\P_t^{\perp}\U_t\U_t^{\top}) \underset{(a)}{=} \trace\left({\P_t^{\perp}\E_t[r_t\u_{t,-}\u_{t,-}^{\top}]}\right) = \E_{t}\left[{r_t\trace\left({\P_t^{\perp}\u_{t,-}\u_{t,-}^{\top}}\right)}\right] \\
&= r_t\E_{t}[\Vert{\P_t^{\perp}\u_{t,-}}\Vert^2],
\end{align*}
where (a) holds since by definition of $\u_{t,-}$ in Algorithm \ref{alg:AFW} we have that $\E_t[\u_{t,-}\u_{t,-}^{\top}] = \frac{1}{r_t}\U_t\U_t^{\top}$.
\end{proof}

The following lemma is central to the analysis of the benefit of both standard Frank-Wolfe steps and away steps in Algorithm \ref{alg:AFW}. It lower-bounds the eigengap between the extreme eigenvalues associated with the principal $r^*$-dimensional subspace (corresponding to the $r^*$ lowest eigenvalues) of the current gradient direction $\nabla_t$. This is essentially a type of Polyak-\L{}ojasiewicz (PL) inequality w.r.t. the face of $\mS^n$ corresponding to the principal subspace of $\nabla_t$. 
\begin{lemma}[in-face eigengap]\label{lem:eigengap}
Assume $r^* \geq 2$. Fix some iteration $t$ of Algorithm \ref{alg:AFW} and suppose that $\trace(\P_t^{\perp}\X_t) \leq \frac{h_t}{4\Vert{\nabla_t}\Vert}$. Then,
%Assume $r^* \geq 2$. Fix some iteration $t$ of Algorithm \ref{alg:AFW} and suppose that $\rank(\X_t) \leq r^*$ and $\trace(\P_t^{\perp}\X_t) \leq \frac{h_t}{4\Vert{\nabla_t}\Vert}$. Then,
\begin{align*}
%\lambda_{n-r^*+1}(\nabla_t) - \lambda_n(\nabla_t) \geq \frac{h_t}{2\sqrt{r^*}\Vert{\X_t-\X^*}\Vert_F}.
\lambda_{n-r^*+1}(\nabla_t) - \lambda_n(\nabla_t) \geq \sqrt{\frac{\alpha{}h_t}{8r^*}}.
\end{align*}
\end{lemma}
\begin{proof}
Denote $\X^*=\arg\min_{\Y\in\mX^*}\Vert{\Y-\X_t}\Vert_F$. Denote $\W = \P_t(\X_t-\X^*)\P_t$, $\W^{\perp} = \P_t^{\perp}(\X_t-\X^*)\P_t^{\perp}$, and note that since $\rank(\P_t) = r^*$ it follows that $\rank(\W) \leq r^*$.
%Suppose $\lambda_{n-r+1}(\nabla_t) - \lambda_n(\nabla_t) = O(\sqrt{h_t})$ --- otherwise, we can take a linear convergence-type step.
Using the convexity of $f$ we have that,
\begin{align}\label{eq:lem:eigengap:1}
h_t &\leq \langle{\X_t-\X^*,\nabla_t}\rangle  = \langle{\W,\nabla_t}\rangle + \langle{\W^{\perp},\nabla_t}\rangle.
\end{align}

Note that by definition of $\P_t$ and $\W$, for any eigenvector $\u$ of $\W$ corresponding to a non-zero eigenvalue, we have that
\begin{align}\label{eq:lem:eigengap:101}
\lambda_n(\nabla_t)=\min_{\v: \P_t\v = \v, \Vert{\v}\Vert=1}\v^{\top}\nabla_t\v \leq \u^{\top}\nabla_t\u \leq \max_{\v: \P_t\v = \v, \Vert{\v}\Vert=1}\v^{\top}\nabla_t\v = \lambda_{n-r^*+1}(\nabla_t).
\end{align}

Thus, we have that,
\begin{align}\label{eq:lem:eigengap:2}
\langle{\W,\nabla_t}\rangle &\leq \left({\sum_{i\in[n]: \lambda_i(\W)>0}\lambda_i(\W)}\right)\lambda_{n-r^*+1}(\nabla_t) \nonumber  \\
&~~~+ \left({\trace(\W) - \sum_{i\in[n]:\lambda_i(\W)>0}\lambda_i(\W)}\right)\lambda_{n}(\nabla_t) \nonumber\\
& \underset{(a)}{\leq} \sqrt{r^*}\Vert{\W}\Vert_F\left({\lambda_{n-r^*+1}(\nabla_t) - \lambda_n(\nabla_t)}\right)  + \trace(\W)\lambda_n(\nabla_t) \nonumber\\
& \underset{(b)}{\leq} \sqrt{r^*}\Vert{\X_t-\X^*}\Vert_F\left({\lambda_{n-r^*+1}(\nabla_t) - \lambda_n(\nabla_t)}\right)   + \trace(\W)\lambda_n(\nabla_t) \nonumber\\
& \underset{(c)}{=} \sqrt{r^*}\Vert{\X_t-\X^*}\Vert_F\left({\lambda_{n-r^*+1}(\nabla_t) - \lambda_n(\nabla_t)}\right)   - \trace(\W^{\perp})\lambda_n(\nabla_t) \nonumber\\
&= \sqrt{r^*}\Vert{\X_t-\X^*}\Vert_F\left({\lambda_{n-r^*+1}(\nabla_t) - \lambda_n(\nabla_t)}\right) \nonumber \\
&~~~+ \left({\trace(\P_t^{\perp}\X^*) - \trace(\P_t^{\perp}\X_t)}\right)\lambda_n(\nabla_t),
\end{align}
where  (a) holds since 
\begin{align*}
\sum_{i\in[n]: \lambda_i(\W)>0}\lambda_i(\W) &\leq \sum_{i\in[n]}\vert{\lambda_i(\W)}\vert \leq \sqrt{\rank(\W)}\sqrt{\sum_{i\in[n]}\lambda_i(\W)^2} \\
&\leq \sqrt{r^*}\Vert{\W}\Vert_F,
\end{align*}
(b) holds since $\Vert{\W}\Vert_F \leq \Vert{\X_t-\X^*}\Vert_F$  which is due to the inequality $\Vert{\A\B}\Vert_F \leq \Vert{\A}\Vert\cdot\Vert{\B}\Vert_F$ for any two matrices $\A,\B\in\mbS^n$, and (c) holds since 
\begin{align*}
\trace(\W) &= \trace(\P_t(\X_t-\X^*)) = \trace((\I-\P_t^{\perp})(\X_t-\X^*)) \\
&= \trace(\X_t-\X^*) - \trace(\P_t^{\perp}(\X_t-\X^*)) = -\trace(\P_t^{\perp}(\X_t-\X^*)) = -\trace(\W^{\perp}).
\end{align*}

Similarly to Eq. \eqref{eq:lem:eigengap:101}, by the definition of $\P_t^{\perp}$, we have that
\begin{align*}
\lambda_{n-r^*}(\nabla_t) = \min_{\v:\P_t^{\perp}\v=\v, \Vert{\v}\Vert=1}\v^{\top}\nabla_t\v \leq \max_{\v:\P_t^{\perp}\v=\v, \Vert{\v}\Vert=1}\v^{\top}\nabla_t\v = \lambda_1(\nabla_t). 
\end{align*}
Thus, using the definition of $\W^{\perp}$, we have that
\begin{align}\label{eq:lem:eigengap:3}
\langle{\W^{\perp},\nabla_t}\rangle &= \langle{\P_t^{\perp}\X_t\P_t^{\perp},\nabla_t}\rangle - \langle{\P_t^{\perp}\X^*\P_t^{\perp},\nabla_t}\rangle \nonumber \\
&\leq \trace(\P_t^{\perp}\X_t)\lambda_1(\nabla_t) - \trace(\P_t^{\perp}\X^*)\lambda_{n-r^*}(\nabla_t),
\end{align}
where the inequality uses the fact that $\P_t^{\perp}\X_t\P_t^{\perp}$ and $\P_t^{\perp}\X^*\P_t^{\perp}$ are positive semidefinite.

Combining \eqref{eq:lem:eigengap:1}, \eqref{eq:lem:eigengap:2} and \eqref{eq:lem:eigengap:3} gives,
\begin{align*}
h_t &\leq \sqrt{r^*}\Vert{\X_t-\X^*}\Vert_F\left({\lambda_{n-r^*+1}(\nabla_t) - \lambda_n(\nabla_t)}\right) + 2\trace(\P_t^{\perp}\X_t)\Vert{\nabla_t}\Vert \\
&~~~+ \trace(\P_t^{\perp}\X^*)\left({\lambda_n(\nabla_t) - \lambda_{n-r^*}(\nabla_t)}\right) \\
&\leq \sqrt{r^*}\Vert{\X_t-\X^*}\Vert_F\left({\lambda_{n-r^*+1}(\nabla_t) - \lambda_n(\nabla_t)}\right) + 2\trace(\P_t^{\perp}\X_t)\Vert{\nabla_t}\Vert.
\end{align*}
Rearranging the inequality, using the assumption on $\trace(\P_t^{\perp}\X_t)$, and the quadratic growth bound $\Vert{\X_t-\X^*}\Vert_F \leq \sqrt{\frac{2h_t}{\alpha}}$, yields the lemma.

\end{proof}

The following Lemma, which is motivated by the result of the previous Lemma \ref{lem:eigengap}, establishes the reduction in objective function value for FW / away steps, under a positive in-face eigengap. 
\begin{lemma}[FW/away steps under in-face eigengap]\label{lem:FWAFWstep}
Assume $r^* \geq 2$. Fix some iteration $t$ of Algorithm \ref{alg:AFW}  and suppose that $r_t \geq r^*$ and $\lambda_{r_t}(\X_t) \geq \frac{\delta_{t,r^*}}{16\beta}$. Then,
%Assume $r^* \geq 2$. Fix some iteration $t$ of Algorithm \ref{alg:AFW}. Denote $\delta_{t,r^*} = \lambda_{n-r^*+1}(\nabla_t) - \lambda_n(\nabla_t)$ and $r_t = \rank(\X_t)$, and suppose that $r_t \geq r^*$ and $\lambda_{r_t}(\X_t) \geq \frac{\delta_{t,r^*}}{16\beta}$. It holds that
\begin{align}\label{eq:lem:fwafw:res}
\min\{f(\X_{t+1}^{\textrm{FW}}), f(\X_{t+1}^{\textrm{AFW}})\} \leq f(\X_t) - \frac{\delta_{t,r^*}^2}{64\beta}.
\end{align}
\end{lemma}
\begin{proof}
Since $\rank(\X_t) \geq r^*$, according to the Poincar\'{e} inequality (e.g., Theorem 4.4 in \cite{calafiore2014optimization}, see also Footnote \ref{footnote:pi}), we have that there exists a unit vector $\v\in\textrm{Im}(\X_t)$ such that $\v^{\top}\nabla_t\v \geq \lambda_{n-r^*+1}(\nabla_t)$, which in turn implies that $\v_{t,-}^{\top}\nabla_t\v_{t,-} \geq \lambda_{n-r^*+1}(\nabla_t)$. Also, we clearly have that $\v_{t,+}^{\top}\nabla_t\v_{t,+} = \lambda_n(\nabla_t)$. This gives,
\begin{align}\label{lem:fwafw:eq:0}
\max\{\langle{\X_t - \v_{t,+}\v_{t,+}^{\top}, \nabla_t}\rangle, \langle{\v_{t,-}\v_{t,-}^{\top}-\X_t,\nabla_t}\rangle\} &\geq 
\frac{1}{2}\langle{\v_{t,-}\v_{t,-}^{\top} - \v_{t,+}\v_{t,+}^{\top},\nabla_t}\rangle \nonumber \\
&\geq \frac{\delta_{t,r^*}}{2},
\end{align}
Using the smoothness of $f$ we have that for any step-size $\eta\in[0,1]$,
\begin{align}\label{lem:fwafw:eq:1}
f\left({(\X_t + \eta(\v_{t,+}\v_{t,+}^{\top}-\X_t)}\right) &\leq f(\X_t) + \eta\langle{\v_{t,+}\v_{t,+}^{\top}-\X_t,\nabla_t}\rangle \nonumber \\
&~~~+ \frac{\eta^2\beta}{2}\Vert{\v_{t,+}\v_{t,+}^{\top}-\X_t}\Vert_F^2 \nonumber \\
&\leq  f(\X_t) + \eta\langle{\v_{t,+}\v_{t,+}^{\top}-\X_t,\nabla_t}\rangle  + \eta^2\beta.
\end{align}
Similarly, we have that for any step-size $\eta\in[0,\lambda_{r_t}(\X_t)]\subseteq[0,(\v_{t,-}^{\top}\X_t^{\dagger}\v_{t,-})^{-1}]$ it holds that,
\begin{align}\label{lem:fwafw:eq:2}
f\left({\frac{1}{1-\eta}\left({\X_t - \eta\v_{t,-}\v_{t,-}^{\top}}\right)}\right) &= f\left({\X_t + \frac{\eta}{1-\eta}\left({\X_t - \v_{t,-}\v_{t,-}^{\top}}\right)}\right) \nonumber \\
& \leq f(\X_t) + \frac{\eta}{1-\eta}\langle{\X_t-\v_{t,-}\v_{t,-}^{\top}, \nabla_t}\rangle + \frac{\eta^2\beta}{(1-\eta)^2} \nonumber \\
& \leq f(\X_t) + \eta\langle{\X_t-\v_{t,-}\v_{t,-}^{\top}, \nabla_t}\rangle + \frac{\eta^2\beta}{(1-\eta)^2},
\end{align}
where in the last inequality we have used the fact that $\eta\in[0,1)$ and $\langle{\X_t-\v_{t,-}\v_{t,-}^{\top}, \nabla_t}\rangle \leq 0$.

Combining Eq. \eqref{lem:fwafw:eq:1} and Eq. \eqref{lem:fwafw:eq:2} with Eq. \eqref{lem:fwafw:eq:0}, and the fact that the step sizes $\eta_t^{\textrm{FW}}, \eta_t^{\textrm{AFW}}$ in Algorithm \ref{alg:AFW} are set via line-search, we have that 
\begin{align}\label{lem:fwafw:eq:3}
\forall\eta\in[0,\lambda_{r_t}(\X_t)]:~ \min\{f(\X_{t+1}^{\textrm{FW}}), f(\X_{t+1}^{\textrm{AFW}})\} \leq f(\X_t) - \frac{\eta\delta_{t,r^*}}{2} +\frac{\eta^2\beta}{(1-\eta)^2}. 
\end{align}
In particular, under the assumptions of the lemma that $\lambda_{r_t}(\X_t) \geq \frac{\delta_{t,r^*}}{16\beta}$ and $r_t \geq r^* \geq 2$, which implies that $\lambda_{r_t}(\X_t) \leq 1/2$,  setting $\eta = \frac{\delta_{t,r^*}}{16\beta}$ in \eqref{lem:fwafw:eq:3}  we have that,
\begin{align*}
\min\{f(\X_{t+1}^{\textrm{FW}}), f(\X_{t+1}^{\textrm{AFW}})\} \leq f(\X_t) - \frac{\delta_{t,r^*}^2}{32\beta} +\frac{\delta_{t,r^*}^2}{64\beta} = f(\X_t) - \frac{\delta_{t,r^*}^2}{64\beta},
\end{align*}
where we have used the fact that with our choice of $\eta$, which satisfies $\eta \leq 1/2$, we have that $\frac{\eta^2\beta}{(1-\eta)^2} \leq \frac{\eta^2\beta}{(1-1/2)^2} = \frac{\delta_{t,r^*}^2}{64\beta}$.
\end{proof}

We are now finally positioned to complete the proof of Theorem \ref{thm:main}.

\begin{proof}[Proof of Theorem \ref{thm:main}]
First note that Result \eqref{eq:mainthm:res:0} (monotonicity of the sequence $(h_t)_{t\geq 1}$) follows in a straightforward manner from the design of Algorithm \ref{alg:AFW}: on each iteration $t$ a drop step is taken only if it does not increase the objective function value, and if a drop step is not taken then, due to the update rule for the next iterate $\X_{t+1}$ which suggests that $f(\X_{t+1}) \leq f(\X_{t+1}^{\textrm{FW}})$ and due to the use of line-search in the computation of $f(\X_{t+1}^{\textrm{FW}})$, we have that $f(\X_{t+1}) \leq f(\X_{t+1}^{\textrm{FW}})  \leq f(\X_t)$.

We turn to prove Result  \eqref{eq:mainthm:res:1}. We use again the observation that on each iteration $t$ which is not a drop step Algorithm \ref{alg:AFW} reduces the objective function value at least by the amount that a standard Frank-Wolfe step with line-search will reduce ($f(\X_{t+1}) \leq f(\X_{t+1}^{\textrm{FW}})$). Combining this observation with  Result   \eqref{eq:mainthm:res:0}, we have that on each iteration $t$ of Algorithm \ref{alg:AFW}, denoting by $t_{\textrm{drop}}$ the overall number of drop steps performed up to iteration $t$, the approximation error $h_t$ is upper-bounded by $h_{t-t_{\textrm{drop}}}^{\textrm{FW}}$ --- the approximation error of the standard FW with line search  at the beginning of its $(t-t_{\textrm{drop}})$ iteration. Result \eqref{eq:mainthm:res:1} now follows from using the well-known convergence result for the standard Frank-Wolfe with line-search over the spectrahedron $\mS^n$ (i.e., $h_t^{\textrm{FW}} \leq \frac{4\beta}{t+1}$ for all $t\geq 2$, see for instance \cite{jaggi2013revisiting}) and the upper-bound on $t_{\textrm{drop}}$ in Observation \ref{obsrv:dropsteps}.

We now turn to prove Result \eqref{eq:mainthm:res:2} which deals with the specific case $r^*=1$, i.e.,  a rank-one optimal solution. In this case, under Assumption \ref{ass:sc} there exists a unique optimal solution $\X^*$, see \cite{garber2023linear}.  \cite{garber2023linear} proved  (see proof of Theorem 2 in  \cite{garber2023linear}) that on each iteration $t$ for which we have $\Vert{\nabla_t -\nabla{}f(\X^*)}\Vert_F \leq \frac{\delta}{3}$, it holds that 
\begin{align}\label{eq:thm:main:0}
f(\X_{t+1}^{\textrm{FW}}) - f^* \leq \left({f(\X_t) - f^*}\right)\left({1-\min\{\frac{\delta}{12\beta}, \frac{1}{2}\}}\right).
\end{align}
Using the smoothness and convexity of $f$ we have that (see for instance Lemma \ref{lem:gradDist} in the appendix),
\begin{align*}
\Vert{\nabla_t -\nabla{}f(\X^*)}\Vert_F \leq \sqrt{\beta{}h_t}.
\end{align*}
Thus, whenever $h_t \leq \frac{\delta^2}{9\beta}$, we indeed have that $\Vert{\nabla_t -\nabla{}f(\X^*)}\Vert_F \leq \frac{\delta}{3}$.
Combining this with Eq. \eqref{eq:thm:main:0} and the fact that in case a drop step is not taken we have that $f(\X_{t+1}) \leq f(\X_{t+1}^{\textrm{FW}})$, we obtain Result \eqref{eq:mainthm:res:2}.\\

We now move on to prove the main result of the theorem --- Result \eqref{eq:mainthm:res:3}. Recall in this case we assume $n > r^* \geq 2$.
Let us fix some iteration $t$ of Algorithm \ref{alg:AFW} and denote $\X^*=\arg\min_{\Y\in\mX^*}\Vert{\Y-\X_t}\Vert_F$. Using Weyl's inequality for the eigenvalues, we have that
\begin{align}\label{eq:thm:main:3}
\forall i\in[n]:\quad \vert{\lambda_i(\nabla_t) - \lambda_i(\nabla{}f(\X^*))}\vert &\leq \Vert{\nabla_t - \nabla{}f(\X^*)}\Vert \leq \sqrt{\beta{}h_t},
\end{align}
where the last inequality is due to the smoothness and convexity of $f$ (see Lemma \ref{lem:gradDist} in the appendix).

Thus, if $h_t \leq \epsilon_1 := \frac{\delta^2}{16\beta}$, we have that
\begin{align}\label{eq:thm:main:1}
\delta_t  := \lambda_{n-r^*}(\nabla_t) - \lambda_{n-r^*+1}(\nabla_t) \geq  \delta -2\sqrt{\beta\epsilon_1} = \frac{\delta}{2}.
\end{align}

Denote $\epsilon_2 = \min\{\frac{(\delta/2)^2}{18\beta}, \frac{\delta/2}{6}\}$. Using Lemma \ref{lem:dropstep}   we have that if $h_t < \min\{\epsilon_1, \epsilon_2\}$ and  $\rank(\X_t)>r^*$, then a drop step will be performed. Note also that if  $h_t < \min\{\epsilon_1, \epsilon_2\}$ and a drop step is not taken, then it must hold that $\rank(\X_t) \leq r^*$. Note also that if $h_t < \epsilon_3 := \frac{\alpha\lambda_{r^*}^2}{8}$, then using Weyl's inequality again and the quadratic growth assumption we have that, 
\begin{align*}
\vert{\lambda_{r^*}(\X_t) - \lambda_{r^*}(\X^*)}\vert \leq \Vert{\X_t-\X^*}\Vert_F \leq \sqrt{\frac{2h_t}{\alpha}},
\end{align*}
%\begin{align*}
%\frac{\alpha\lambda_{r^*}(\X^*)^2}{8} > h_t \geq \frac{\alpha}{2}\Vert{\X_{t}-\X^*}\Vert_F^2 \geq \frac{\alpha}{2}\left({\lambda_{r^*}(\X_{t}) - \lambda_{r^*}(\X^*)}\right)^2,
%\end{align*}
which implies that
\begin{align}\label{eq:thm:main:2}
% 3\lambda_{r^*}(\X^*)/2 \geq  \lambda_{r^*}(\X_{t}) \geq \lambda_{r^*}(\X^*)/2,
  \lambda_{r^*}(\X_{t}) \geq \frac{\lambda_{r^*}}{2},
\end{align}
which further implies that $\rank(\X_t) \geq r^*$.
 
Thus, if $h_t < \min\{\epsilon_1,\epsilon_2,\epsilon_3\}$ and a drop step is not performed,  it must hold that $\rank(\X_t) = r^*$.  Furthermore, it holds that $\gamma_t \geq \lambda_{r^*}(\X_t) \geq \lambda_{r^*} /2$.

Define also $\epsilon_4 := \frac{(\delta/2)^3}{24\beta^2}$ and throughout the sequel suppose that indeed  $h_t < \min\{\epsilon_1, \epsilon_2, \epsilon_3, \epsilon_4\}$ and a drop step is not taken. We consider two cases. First, if $\trace(\P_t^{\perp}\X_t) > \frac{h_t}{4G}$, then it follows from Lemma \ref{lem:pfwstep} that
\begin{align*}
\E_t[f(\X_{t+1}^{\textrm{PFW}})] &\leq f(\X_t) - \E_t\left[{  \min\{\frac{\delta_t^2}{9\beta},\frac{\gamma_t\delta_t}{3}\}\Vert{\P_t^{\perp}\u_{t,-}}\Vert^2}\right] \\
&\underset{(a)}{\leq}  f(\X_t) -  \min\{\frac{\delta^2}{36\beta},\frac{\lambda_{r^*}\delta}{12}\}\E_t\left[{\Vert{\P_t^{\perp}\u_{t,-}}\Vert^2}\right] \\
&\underset{(b)}{\leq} f(\X_t)- \frac{\delta}{48Gr^*}\min\{\frac{\delta}{3\beta},\lambda_{r^*}\}h_t,
\end{align*}
where (a) follows from plugging-in \eqref{eq:thm:main:1}, and since  $\gamma_t \geq \lambda_{r^*} / 2$, and (b) follows from Lemma \ref{lem:expectedProj}.

This implies that
\begin{align}\label{eq:thm:main:r1}
\E_t[f(\X_{t+1})] - f^*  &= \E_t\left[{\min\{f(\X_{t+1}^{\textrm{FW}}), f(\X_{t+1}^{\textrm{AFW}}), f(\X_{t+1}^{\textrm{PFW}})\}}\right] - f^* \nonumber \\
&\leq \min\{f(\X_{t+1}^{\textrm{FW}}), f(\X_{t+1}^{\textrm{AFW}}), \E_t\left[{f(\X_{t+1}^{\textrm{PFW}})}\right]\} - f^* \nonumber  \\
&\leq h_t\left({1 - \frac{\delta}{48Gr^*}\min\{\frac{\delta}{3\beta},\lambda_{r^*}\}}\right).
\end{align}

In the complementing case that $\trace(\P_t^{\perp}\X_t) \leq \frac{h_t}{4G}$, it follows from Lemma \ref{lem:eigengap} that 
\begin{align}\label{eq:thm:main:4}
%\lambda_{n-r^*+1}(\nabla_t) - \lambda_n(\nabla_t) \geq \frac{h_t}{2\sqrt{r^*}\Vert{\X_t-\X^*}\Vert_F}.
\lambda_{n-r^*+1}(\nabla_t) - \lambda_n(\nabla_t) \geq \sqrt{\frac{\alpha{}h_t}{8r^*}}.
\end{align}
Additionally, using Eq. \eqref{eq:thm:main:3} again we have that,
\begin{align}\label{eq:thm:main:5}
%\lambda_{n-r^*+1}(\nabla_t) - \lambda_n(\nabla_t) \geq \frac{h_t}{2\sqrt{r^*}\Vert{\X_t-\X^*}\Vert_F}.
\lambda_{n-r^*+1}(\nabla_t) - \lambda_n(\nabla_t) &\leq \lambda_{n-r^*+1}(\nabla{}f(\X^*)) - \lambda_n(\nabla{}f(\X^*)) + 2\sqrt{\beta{}h_t} \nonumber \\
& \underset{(a)}{=}  2\sqrt{\beta{}h_t} \leq 2\sqrt{\beta{}\epsilon_3} = \sqrt{\frac{\alpha\beta}{2}}\lambda_{r^*} \underset{(b)}{\leq} \sqrt{2}\beta\lambda_{r^*}(\X_t),
\end{align}
where (a) follows since due to the optimality of $\X^*$ it holds that $\lambda_n(\nabla{}f(\X^*)) =\lambda_{n-r^*+1}(\nabla{}f(\X^*))$ (see also Lemma 5.2 in \cite{garber2021convergence}), and (b) follows from Eq. \eqref{eq:thm:main:2} and the fact that $\alpha \leq \beta$ (e.g., Lemma \ref{lem:alphabeta} in the appendix).

Thus, combining Lemma \ref{lem:FWAFWstep} with Eq. \eqref{eq:thm:main:4}, we have that
\begin{align}\label{eq:thm:main:r2}
f(\X_{t+1}) -f^* &\leq \min\{f(\X_{t+1}^{\textrm{FW}}), f(\X_{t+1}^{\textrm{AFW}})\} -f^* \leq h_t\left({1- \frac{\alpha}{512\beta{}r^*}}\right).
\end{align}

Combining \eqref{eq:thm:main:r1} and  \eqref{eq:thm:main:r2} yields Result  \eqref{eq:mainthm:res:3}.\\

Finally, it remains to prove Result $\eqref{eq:mainthm:res:4}$ for the case $r^*=n$, which follows as a simplified case of Result  \eqref{eq:mainthm:res:3}. If $h_t < \epsilon_3$, then we have that Eq. \eqref{eq:thm:main:2} and  Eq. \eqref{eq:thm:main:5} hold, and as a consequence the assumptions of Lemma \ref{lem:eigengap} and Lemma \ref{lem:FWAFWstep} hold. Thus, it follows that Eq. \eqref{eq:thm:main:r2} also holds.

%the upper-bound on the overall number of drop steps up to iteration $t$ is due to Observation \ref{obsrv:dropsteps}.
\end{proof}

\section{Numerical Experiments}\label{sec:numerics}
We conduct numerical experiments with synthetic data for the problem of recovering a low-rank ground truth matrix in $\mS^n$ from random linear measurements under two noise models: one in which the measurements are perturbed with isotropic (Gaussian) noise, and we use a standard least squares (LS) estimator. We refer to this as the \textit{LS setting}. In the second model, only a constant fraction of measurements are grossly corrupted, and we use the Huber loss for robust estimation. We refer to this as the \textit{Huber setting}. We use a similar setup and methodology to the one considered in  previous works  \cite{garber2023linear, ding2020spectral}. Concretely, we consider the following optimization problem:
\begin{align}\label{eq:expProb}
\min_{\X\in\mS^n}\{f(\X) := \sum_{i=1}^m\ell\left({\a_i^{\top}(\tau\X)\a_i - \b_i}\right)\},
%\min_{\X\in\mS^n}\{f(\X) := \frac{1}{2}\sum_{i=1}^m\left({\a_i^{\top}(\tau\X)\a_i - \b_i}\right)^2\}.
\end{align}
where in the LS setting $\ell$ is simply the square loss: $\ell(x) = \frac{1}{2}x^2$, and in the Huber setting $\ell(x) = H_{\zeta}(x)$ --- the Huber loss with parameter $\zeta$\footnote{formally defined as: $H_{\zeta}(x) = \left\{ \begin{array}{ll}
         \frac{1}{2}x^2 & \mbox{if $|x| \leq \zeta$};\\
        \zeta(|x|-\frac{1}{2}\zeta) & \mbox{if $|x| > \zeta$}.\end{array} \right.$}.
        
We let $\X_{\sharp} = \U_{\sharp}\U_{\sharp}^{\top}\in\mS^n$ denote the ground truth matrix, where $\U_{\sharp}\in\reals^{n\times r^*}$ is first set to a random matrix with  standard Gaussian entries and then normalized to have unit Frobenius norm. Each $\a_i\in\reals^n$ is a random vector with standard Gaussian entries. We set $\b = \b_{\sharp} + \z$, where $\b_{\sharp}(i) = \a_i^{\top}\X_{\sharp}\a_i$ and $\z$ is additive noise. In the LS setting we set $\z = \frac{\Vert{\b_{\sharp}}\Vert}{2}\v$ for a random unit vector $\v$. In the Huber setting, for each $i\in[m]$ we set with probability $0.15$, $\z(i) = \pm4.2\frac{\Vert{\b_{\sharp}}\Vert}{\sqrt{m}}$, and $\z(i) =0$ with remaining probability. We set the Huber loss parameter to $\zeta = 0.7\frac{\Vert{\b_{\sharp}}\Vert}{\sqrt{m}}$.

In order to avoid over-fitting the noise, in both the LS and Huber settings, we set the scaling parameter $\tau$ in \eqref{eq:expProb} to $\tau = 0.5$ (as also been done in \cite{garber2023linear, ding2020spectral}). %These choices indeed ensure (with high probability) that Problem \eqref{eq:expProb} both recovers the ground truth matrix $X_{\sharp}$ with good accuracy and that strict complementarity  holds (see also \cite{garber2023linear, ding2020spectral}).

We solve Problem \eqref{eq:expProb} to accuracy  of the order $10^{-10}$ or lower and denote our estimate for the optimal solution by $\widehat{\X^*}$ \footnote{this is verified by computing the dual gap for the candidate solution $\widehat{\X^*}$, which is given by $\langle{\widehat{\X^*}, \nabla{}f(\widehat{\X^*})}\rangle - \lambda_n(\nabla{}f(\widehat{\X^*}))$ and is always an upper-bound on the approximation error $f(\widehat{\X^*}) -f^*$ due to the convexity of $f$}. For both the LS and Huber settings we verify that in each of our randomized experiments  the strict complementarity condition indeed holds with substantial positive values (see also Section 2.2.1 in \cite{garber2023linear} for a detailed discussion about verifying the strict complementarity assumption).

We also consider a third setting, which we refer to as the \textit{No SC setting}, in which strict complementarity does not hold. In this setting we do not add noise to the measurements (i.e., we set $\b = \b_{\sharp}$) and we also set $\tau = 1$. This guarantees that at the optimal solution to  \eqref{eq:expProb}, the gradient is zero, and in particular strict complementarity cannot hold.

We set the smoothness parameter $\beta$ to $\beta = n^2/2$ \footnote{while this might seem conservative, note that a simple calculation shows that $\beta \leq \sum_{i=1}^m\Vert{\a_i}\Vert^4$ and thus we should roughly take $\beta \approx mn^2$, however this turns out to be an overly pessimistic estimate (see also experiments in \cite{ding2020spectral})}.

Each one of the figures presented below is the average of 10 i.i.d. runs, where on each run we resample both the data to Problem  \eqref{eq:expProb} and the randomness in our Algorithm \ref{alg:AFW} .

\subsection{Comparison to standard Frank-Wolfe}
We begin by comparing our Algorithm \ref{alg:AFW} (ALG 1 in the figures below) to the standard Frank-Wolfe with line-search method (FW in the figures below). Recall that FW has a worst case convergence rate of $O(1/t)$  \cite{jaggi2010simple}. In \cite{garber2023linear} it was established that under strict complementarity and in the special case $r^*=1$, it enjoys a linear convergence rate, while it remained unknown if this is also the case for $r^*\geq 2$. Figure \ref{fig:OurVsFW}, which plots for each method the value $\log({f(\X_t) - f(\widehat{\X^*})})$ vs the number of iterations $t$, shows that indeed in case $r^*=1$, in both the LS and Huber settings, both FW and our Algorithm \ref{alg:AFW} clearly exhibit a linear convergence rate. However, once we increase $r^*$, we see that while our Algorithm \ref{alg:AFW} maintains a linear convergence rate, FW converges only with a sub-linear rate. Additionally, in the No SC setting, we see that even for $r^*=1$, FW converges sub-linearly, while our algorithm exhibits a linear convergence rate. This suggests that, similarly to the state-of-affairs in the case of optimization over polytopes \cite{garber2020revisiting}, the standard FW does not exhibit linear convergence once the optimal solution is not an extreme point of the feasible set.

\begin{figure}[H]
     \centering
     \begin{subfigure}[t]{0.325\textwidth}
         \centering
         \includegraphics[width=\textwidth]{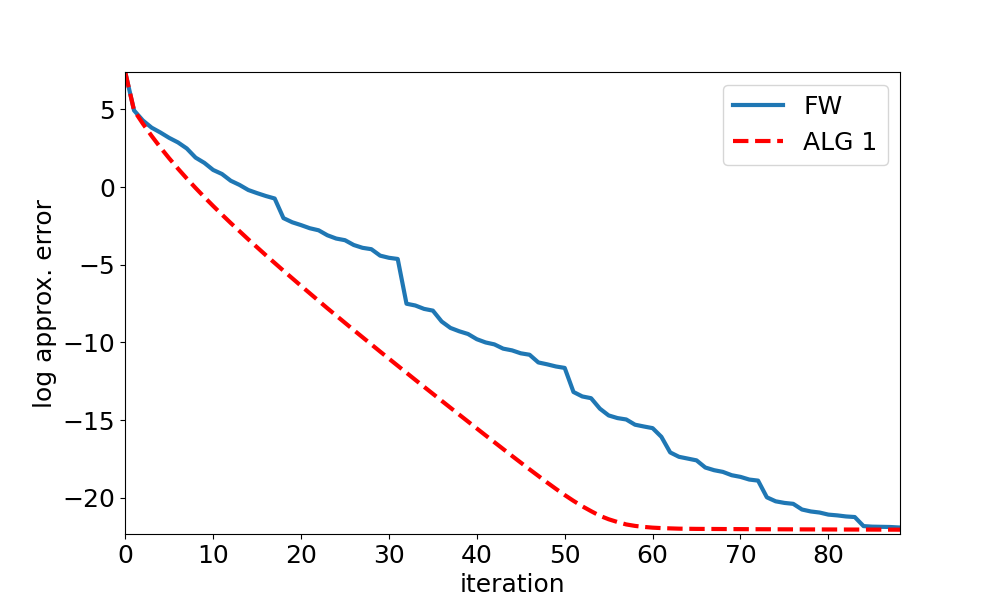}
         \caption{LS, $r^*=1$}
         %\label{fig:y equals x}
     \end{subfigure}
     \begin{subfigure}[t]{0.325\textwidth}
         \centering
         \includegraphics[width=\textwidth]{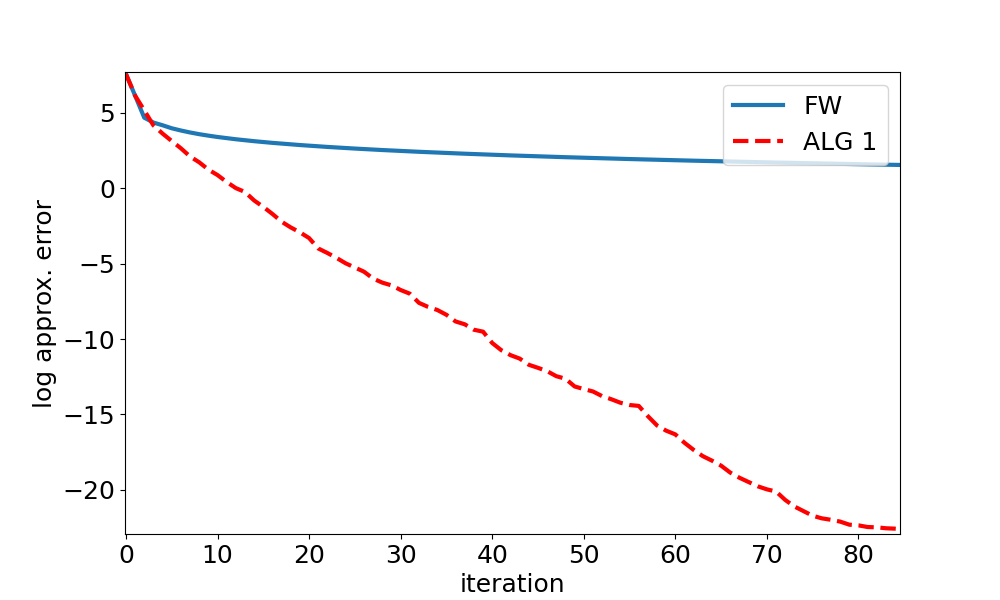}
         \caption{LS, $r^*=2$}
         %\label{fig:three sin x}         
     \end{subfigure}
     \begin{subfigure}[t]{0.325\textwidth}
         \centering
         \includegraphics[width=\textwidth]{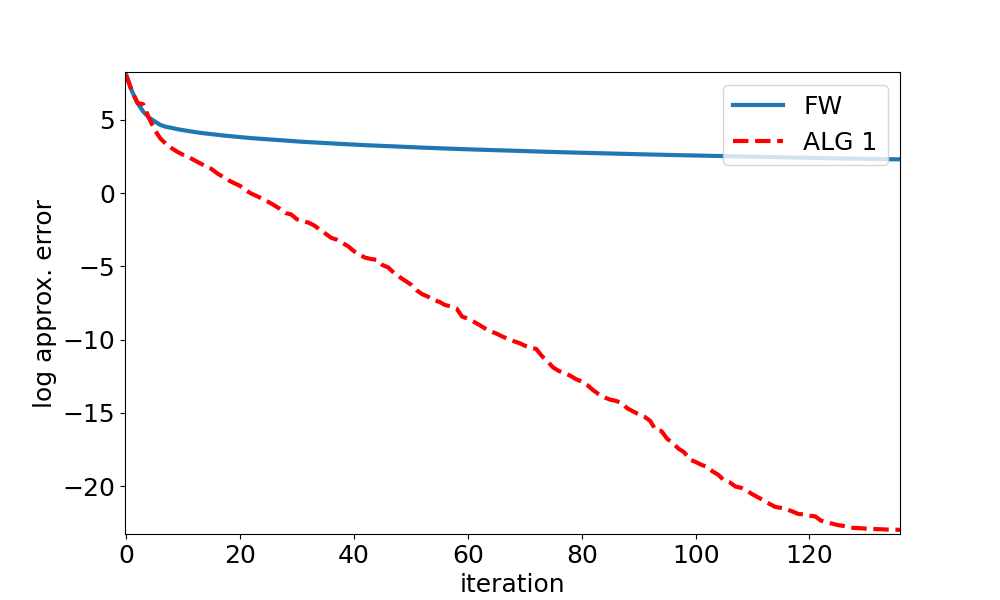}
         \caption{LS, $r^*=5$}
         %\label{fig:three sin x}         
	 \end{subfigure}
	 
     \begin{subfigure}[t]{0.325\textwidth}
         \centering
         \includegraphics[width=\textwidth]{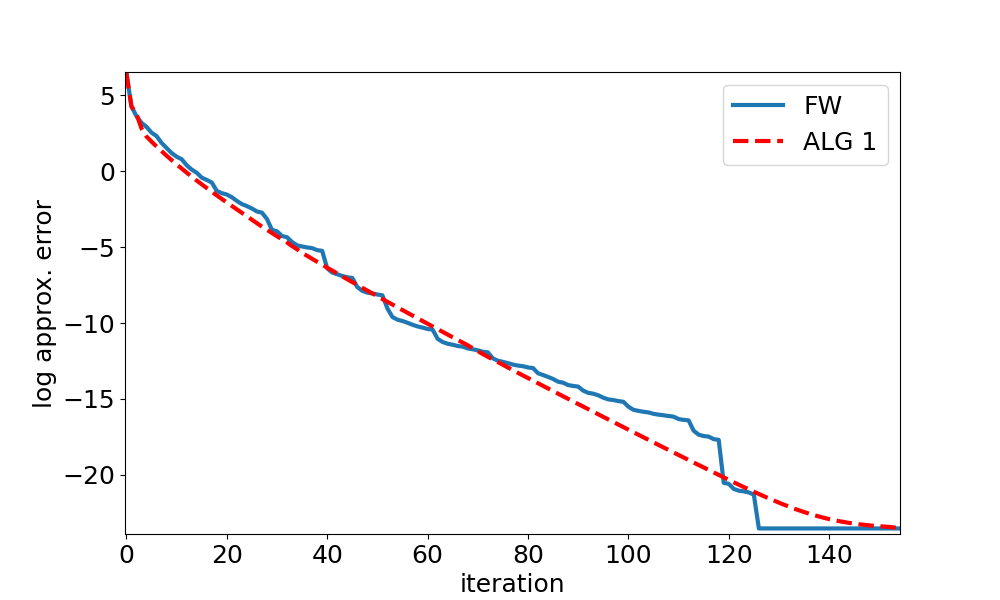}
         \caption{Huber, $r^*=1$}
         %\label{fig:y equals x}
     \end{subfigure}
     \begin{subfigure}[t]{0.325\textwidth}
         \centering
         \includegraphics[width=\textwidth]{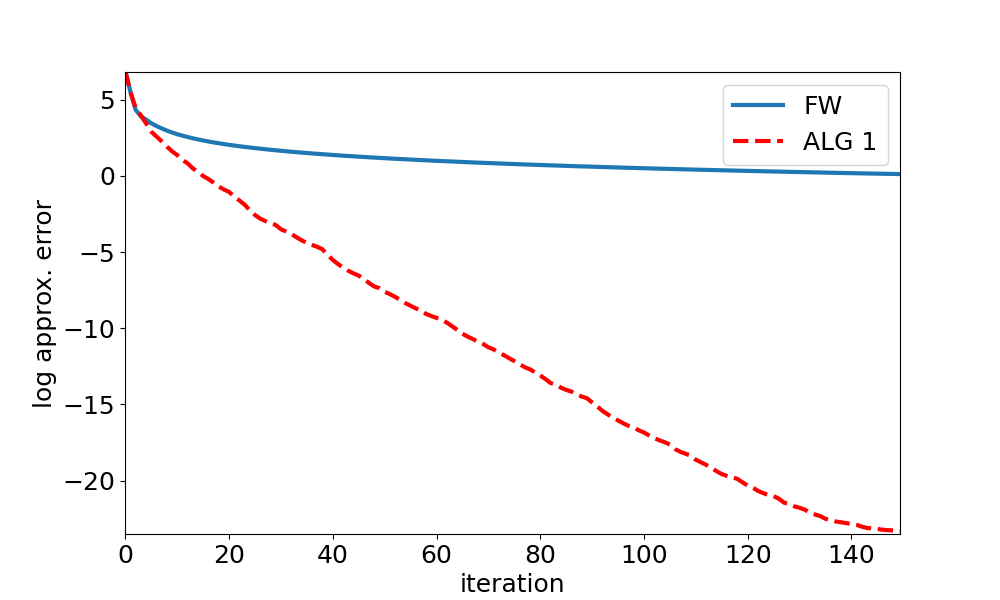}
         \caption{Huber, $r^*=2$}
         %\label{fig:three sin x}         
     \end{subfigure}
     \begin{subfigure}[t]{0.325\textwidth}
         \centering
         \includegraphics[width=\textwidth]{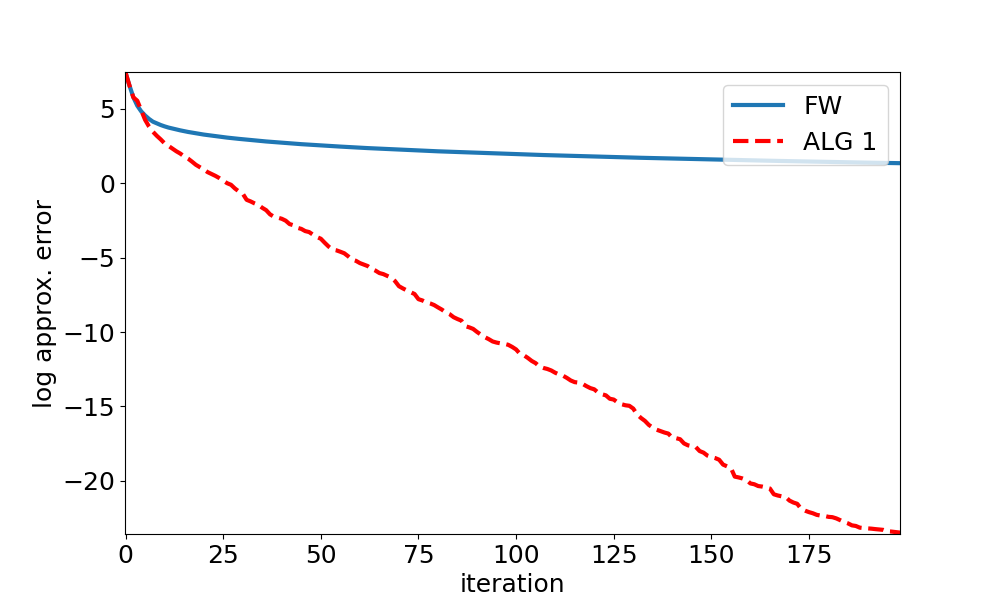}
         \caption{Huber, $r^*=5$}
         %\label{fig:three sin x}   
     \end{subfigure}

     \begin{subfigure}[t]{0.325\textwidth}
         \centering
         \includegraphics[width=\textwidth]{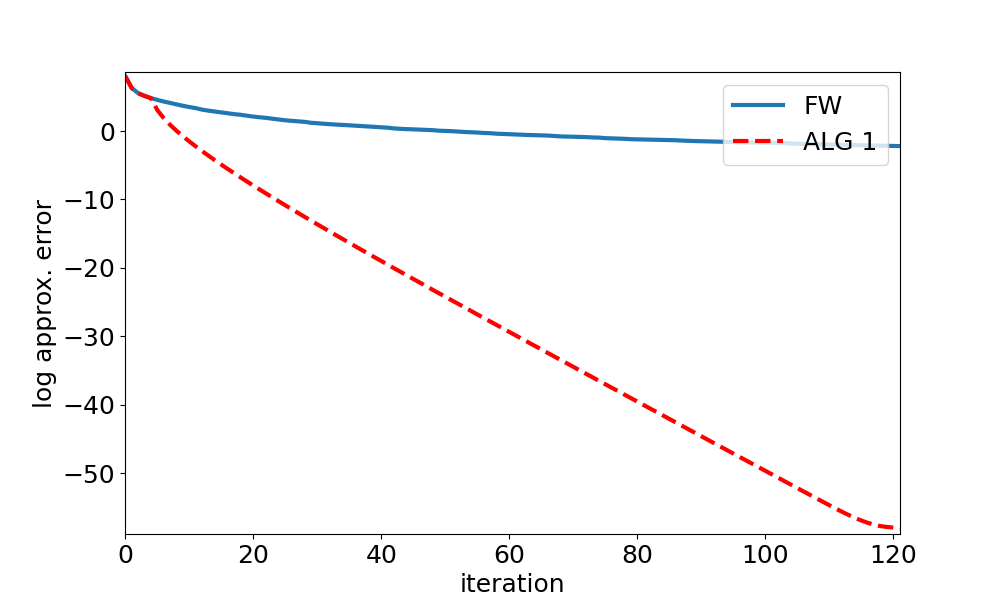}
         \caption{No SC, $r^*=1$}
         %\label{fig:y equals x}
     \end{subfigure}
     \begin{subfigure}[t]{0.325\textwidth}
         \centering
         \includegraphics[width=\textwidth]{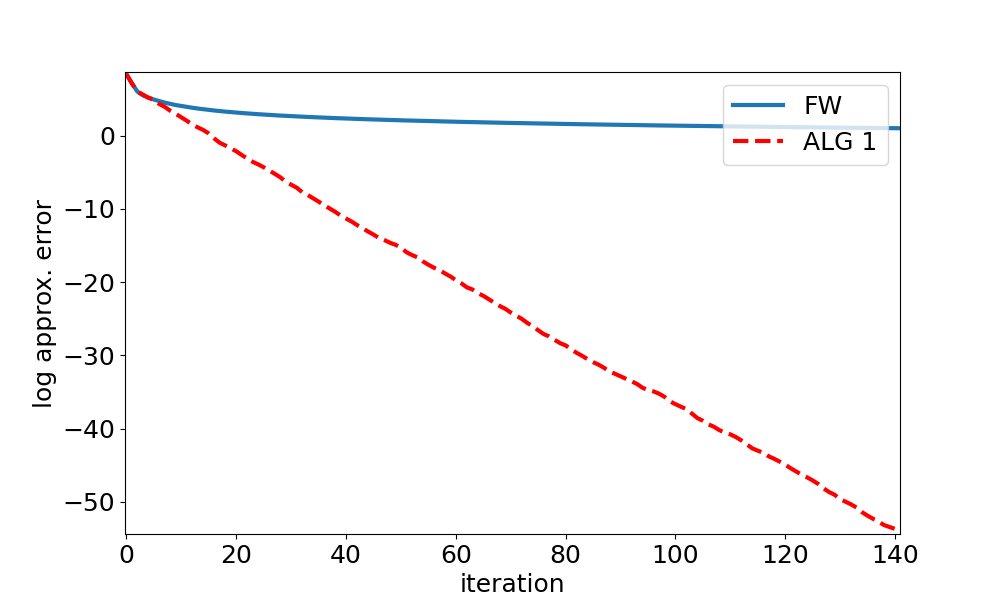}
         \caption{No SC, $r^*=2$}
         %\label{fig:three sin x}         
     \end{subfigure}
     \begin{subfigure}[t]{0.325\textwidth}
         \centering
         \includegraphics[width=\textwidth]{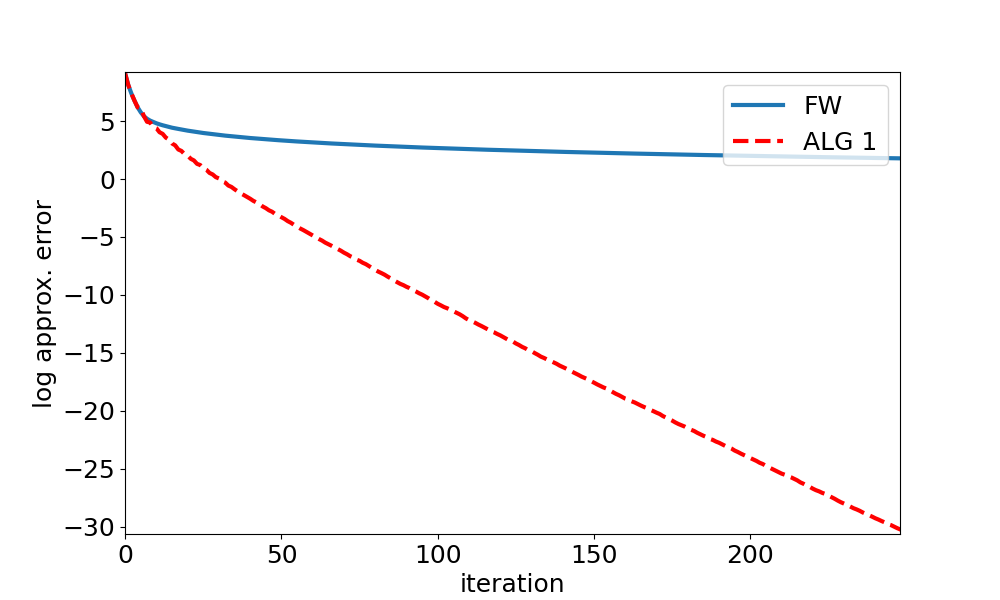}
         \caption{No SC, $r^*=5$}
         %\label{fig:three sin x}   
     \end{subfigure}     

        \caption{Comparison of Frank-Wolfe with line-search and our Algorithm \ref{alg:AFW}. We set $n=100$ and $m=15nr^*$ in all experiments.}
        \label{fig:OurVsFW}
\end{figure}

\subsection{Comparison of variants of Algorithm \ref{alg:AFW}}
We turn to provide empirical motivation for the design choices in our Algorithm \ref{alg:AFW}. In particular, Algorithm \ref{alg:AFW}  prioritizes drop steps over other steps (even if these will reduce the objective value more slowly)  and in addition to away steps (which are a classical concept in Frank-Wolfe-based methods \cite{guelat1986some, lacoste2015global}, though recall our away steps take a different form than those for polytopes),  uses also randomized proximal pairwise steps (which also  require an estimate for the smoothness parameter $\beta$).

Figure \ref{fig:compareVars}  demonstrates how each one of the above mentioned design choices contributes to the fast convergence of our Algorithm \ref{alg:AFW}. In particular, we see that the \textit{ALG 1 - away} variant, which is a natural extension of the \textit{Frank-Wolfe with away steps} method that is known to converge linearly for polytopes \cite{guelat1986some, lacoste2015global, garber2020revisiting}, converges only with a sub-linear rate in our experiments. Interestingly, the \textit{ALG 1 - noaway} variant, which does not use away steps at all (but does use drop steps), performs comparably to  Algorithm \ref{alg:AFW} in both the LS and Huber settings, however, it performs considerably worse in the No SC setting (when $r^*>2$). 

\begin{table*}[h]\renewcommand{\arraystretch}{1.3}
{\small
\begin{center}
  \begin{tabular}{|p{0.18\linewidth} | p{0.75 \linewidth}|} \hline
   algorithm &description \\ \hline
   ALG 1 &  Algorithm \ref{alg:AFW} without  modifications\\ \hline
   ALG 1 - away & Algorithm \ref{alg:AFW} without drop steps and pairwise steps\\ \hline
   ALG 1 - nodrop &Algorithm \ref{alg:AFW} without drop steps \\ \hline
   ALG 1 - det & Deterministic version of Algorithm \ref{alg:AFW}: instead of using the randomized vector $\u_{t,-}$ in the pairwise step on iteration $t$, it is replaced with the deterministic vector  $\v_{t,-}$ (used for the drop and away steps)\\ \hline
   ALG 1 - noaway & Algorithm \ref{alg:AFW} without away steps\\ \hline
  \end{tabular}
\caption{Description of variants of Algorithm \ref{alg:AFW} used in the numerical comparison.}
  \label{table:algorithms}
\end{center}
}
%\vskip -0.2in
\end{table*}\renewcommand{\arraystretch}{1} 

%\begin{figure}[H]
%         \centering
%         \includegraphics[width=0.5\textwidth]{varCompare_r=5.png}     
%        \caption{Comparison of  different variants of our Algorithm  \ref{alg:AFW}. We set $n=100$, $r^*=5$, and $m=15nr^*$.}
%        \label{fig:compareVars}
%\end{figure}

\begin{figure}[H]
     \centering
     \begin{subfigure}[t]{0.325\textwidth}
         \centering
         \includegraphics[width=\textwidth]{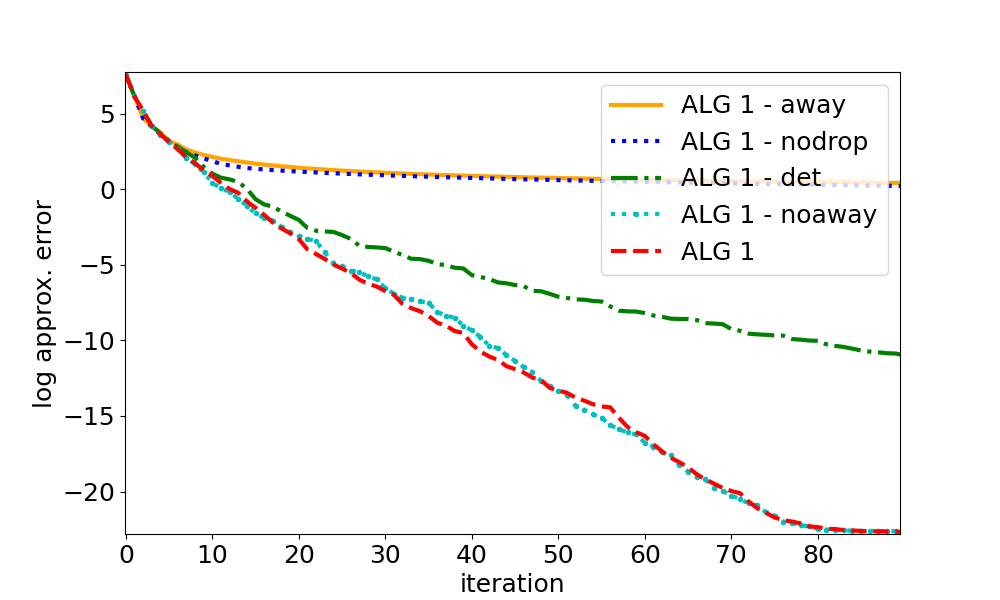}
         \caption{LS, $r^*=2$}
         %\label{fig:y equals x}
     \end{subfigure}
     \begin{subfigure}[t]{0.325\textwidth}
         \centering
         \includegraphics[width=\textwidth]{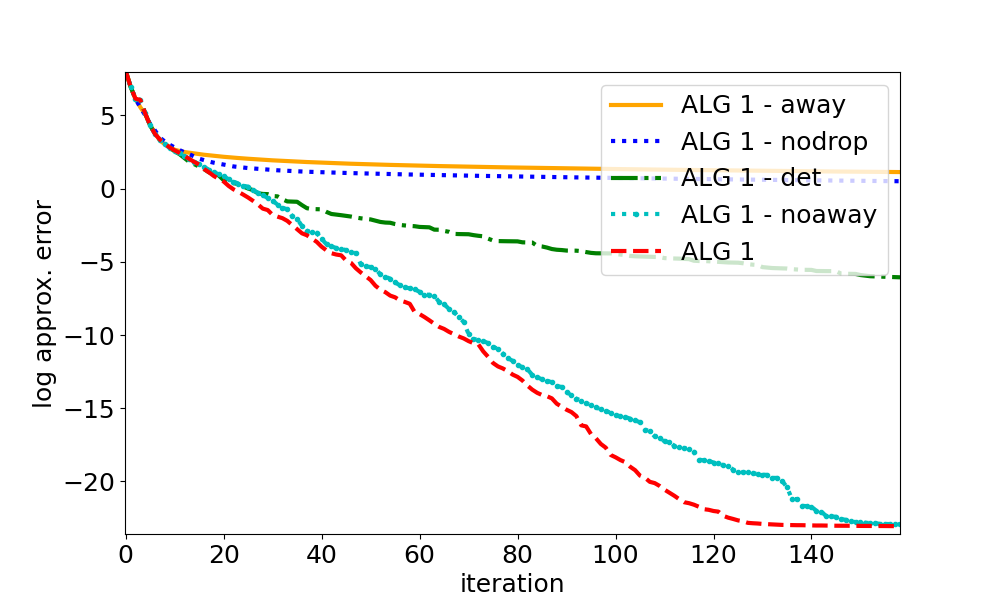}
         \caption{LS, $r^*=5$}
         %\label{fig:three sin x}         
     \end{subfigure}
     \begin{subfigure}[t]{0.325\textwidth}
         \centering
         \includegraphics[width=\textwidth]{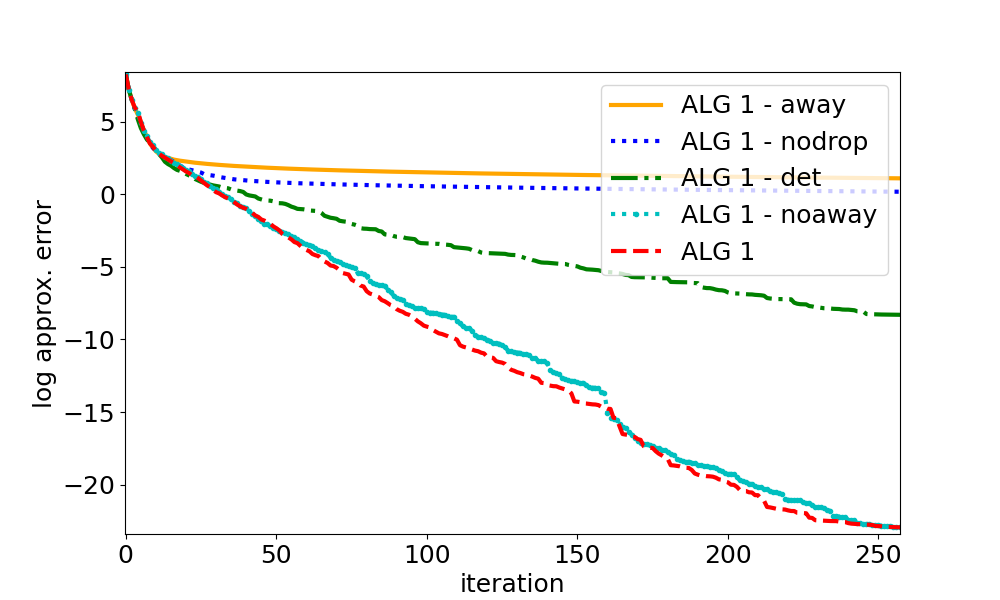}
         \caption{LS, $r^*=7$}
         %\label{fig:three sin x}         
	 \end{subfigure}
	 
     \begin{subfigure}[t]{0.325\textwidth}
         \centering
         \includegraphics[width=\textwidth]{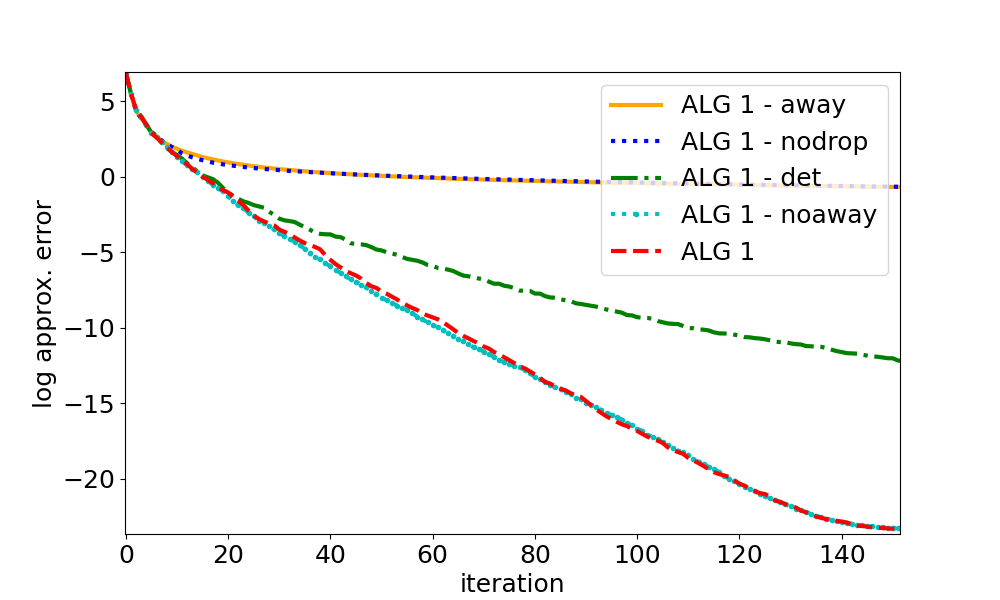}
         \caption{Huber, $r^*=2$}
         %\label{fig:y equals x}
     \end{subfigure}
     \begin{subfigure}[t]{0.325\textwidth}
         \centering
         \includegraphics[width=\textwidth]{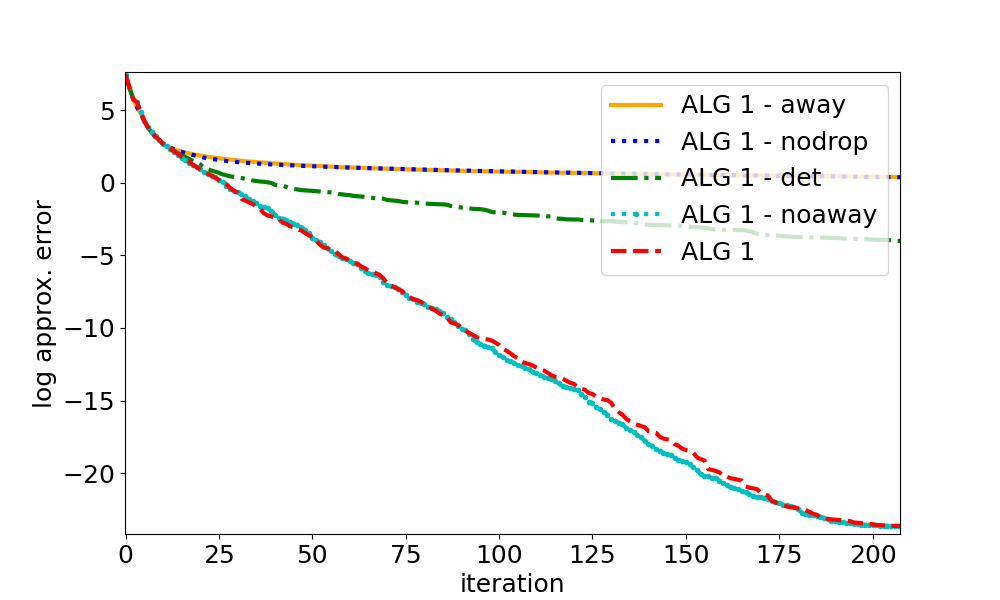}
         \caption{Huber, $r^*=5$}
         %\label{fig:three sin x}         
     \end{subfigure}
     \begin{subfigure}[t]{0.325\textwidth}
         \centering
         \includegraphics[width=\textwidth]{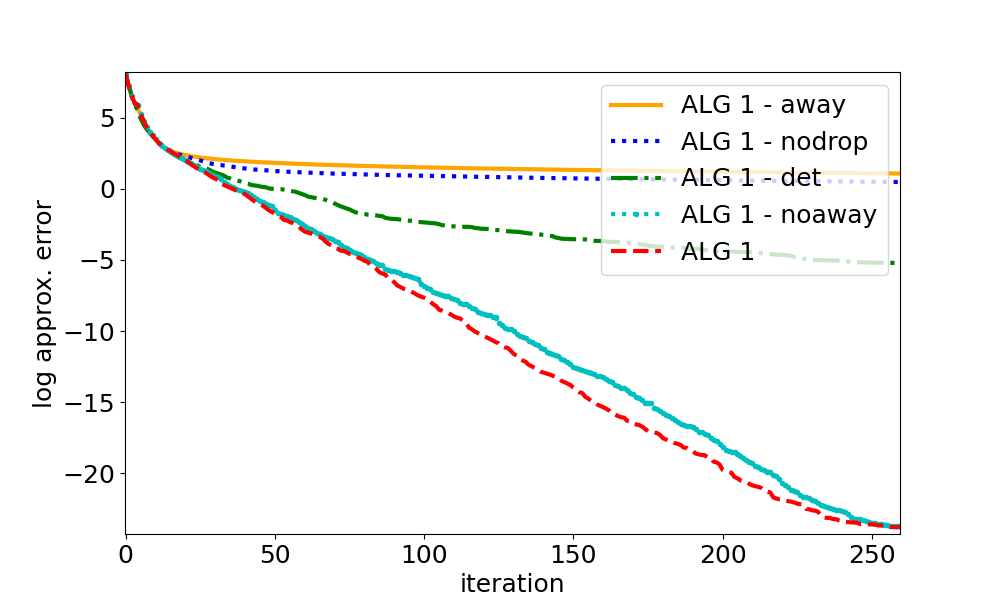}
         \caption{Huber, $r^*=8$}
         %\label{fig:three sin x}   
     \end{subfigure}

     \begin{subfigure}[t]{0.325\textwidth}
         \centering
         \includegraphics[width=\textwidth]{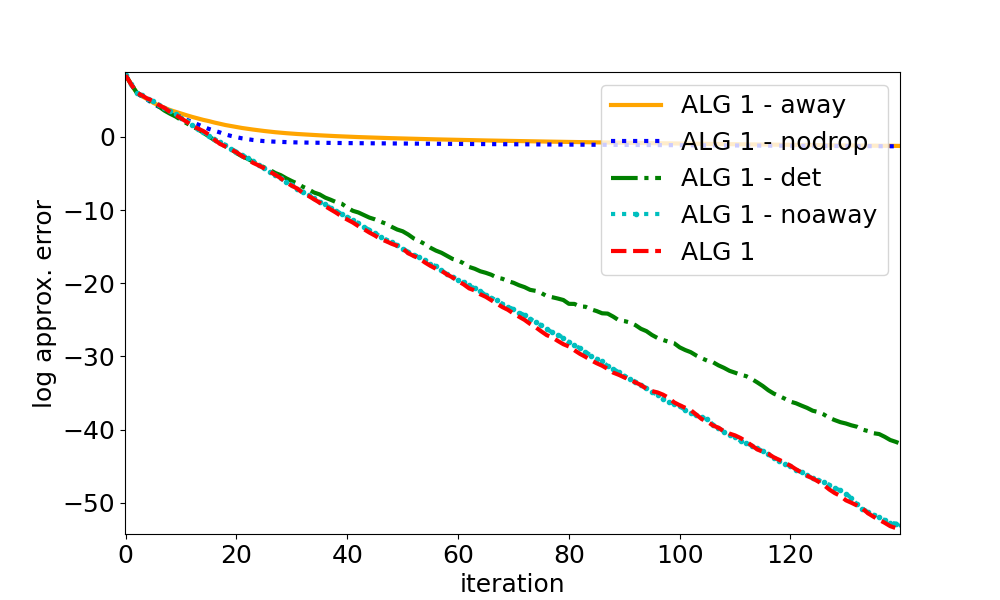}
         \caption{No SC, $r^*=2$}
         %\label{fig:y equals x}
     \end{subfigure}
     \begin{subfigure}[t]{0.325\textwidth}
         \centering
         \includegraphics[width=\textwidth]{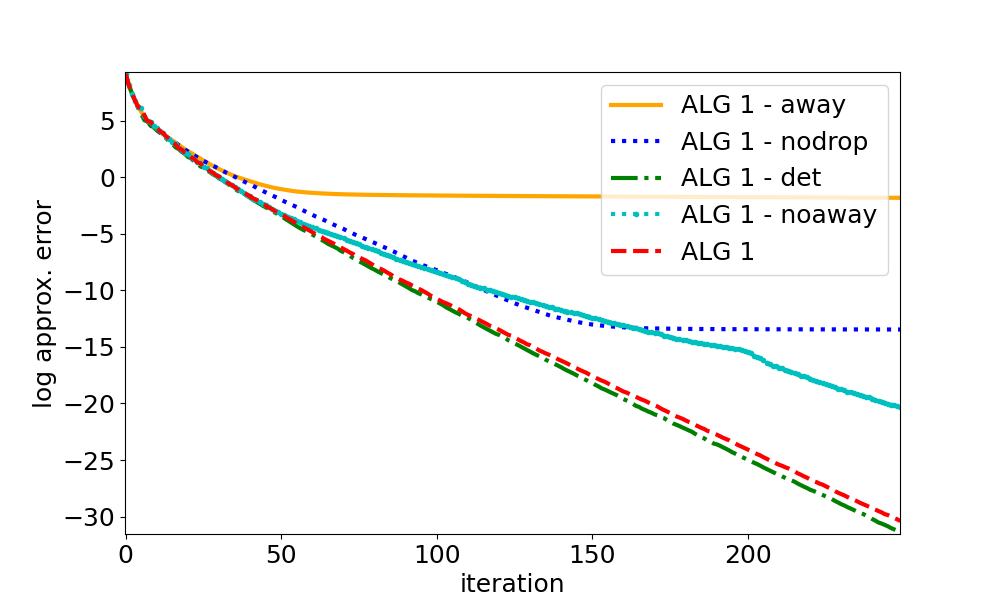}
         \caption{No SC, $r^*=5$}
         %\label{fig:three sin x}         
     \end{subfigure}
     \begin{subfigure}[t]{0.325\textwidth}
         \centering
         \includegraphics[width=\textwidth]{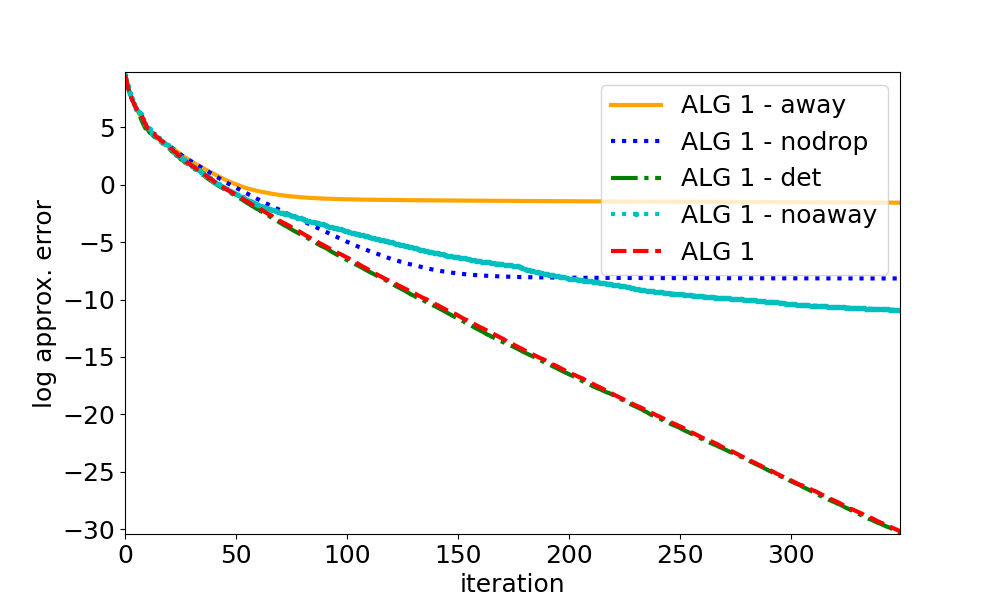}
         \caption{No SC, $r^*=8$}
         %\label{fig:three sin x}   
     \end{subfigure}     

        \caption{Comparison of different variants of our Algorithm \ref{alg:AFW}. In all experiments, except for Huber $r^*=8$, we set $n=100$, $m=15nr^*$. In Huber $r^*=8$ we set $n=100$, $m=20nr^*$.}
        \label{fig:compareVars}
\end{figure}

%\begin{figure}[H]
%     \centering
%     \begin{subfigure}[t]{0.48\textwidth}
%         \centering
%         \includegraphics[width=\textwidth]{varCompare_r=5.png}
%         \caption{$r^*=5$}
%         \label{fig:y equals x}
%     \end{subfigure}
%     \begin{subfigure}[t]{0.48\textwidth}
%         \centering
%         \includegraphics[width=\textwidth]{varCompare_r=10.png}
%         \caption{$r^*=10$}
%         \label{fig:three sin x}         
%     \end{subfigure}
%     
%        \caption{Comparison of  different variants of our Algorithm  \ref{alg:AFW}. We set $n=100$ and $m=15nr^*$.}
%        \label{fig:compareVars}
%\end{figure}

\subsection{Comparison to the Block Frank-Wolfe method}
We also compare our method to Block Frank-Wolfe methods, that is, methods that on each iteration perform an update using a rank-$r$ matrix, for some $r\geq r^*$, rather than a rank-one update as in the standard Frank-Wolfe \cite{allen2017linear, ding2020spectral}. We compare with the method from \cite{allen2017linear} (Block-FW in Figure \ref{fig:compareBlockFW}) which performs updates of the form $\X_{t+1} \gets (1-\eta)\X_t + \eta\V_t$, where $\V_t\in\mS^n$ is a rank-$r$ matrix that is obtained by projecting onto $\mS^n$ the best rank-$r$ approximation (in Frobenius norm) to the matrix $\M_t := \X_t-\frac{1}{\eta\beta}\nabla{}f(\X_t)$, and can be computed using only the top $r$ components in the eigen-decomposition of $\M_t$. While this method converges linearly (under quadratic growth), it requires a tight estimation of $r^*$ (otherwise it suffers from computationally-expensive high-rank eigen-decompositions), and it requires a good estimate of both $\alpha,\beta$ to set the step size to $\eta = O(\alpha/\beta)$, which is difficult to properly tune in practice. Here we consider a somewhat ideal implementation of this method: I. we set $r = r^*$ (i.e., assume exact knowledge of $r^*$), and II. we manually tune the step-size parameter $\eta$ for best performance and set $\eta=0.3$. 

We can see in Figure \ref{fig:compareBlockFW} (left panels) that when comparing the (log) approximation error vs. number of iterations, as expected, the Block-FW method converges faster than our Algorithm \ref{alg:AFW} which only uses rank-one matrix computations. However, when we examine the approximation error vs. the number of rank-one updates (recall our implementation of the Block-FW method does a rank-$r^*$ update on each iteration, while our Algorithm \ref{alg:AFW} does at most a rank-two update), see right panels in Figure \ref{fig:compareBlockFW}, we see that our Algorithm \ref{alg:AFW} is in fact faster.
\begin{figure}[H]
     \centering
     \begin{subfigure}[t]{0.32\textwidth}
         \centering
         \includegraphics[width=\textwidth]{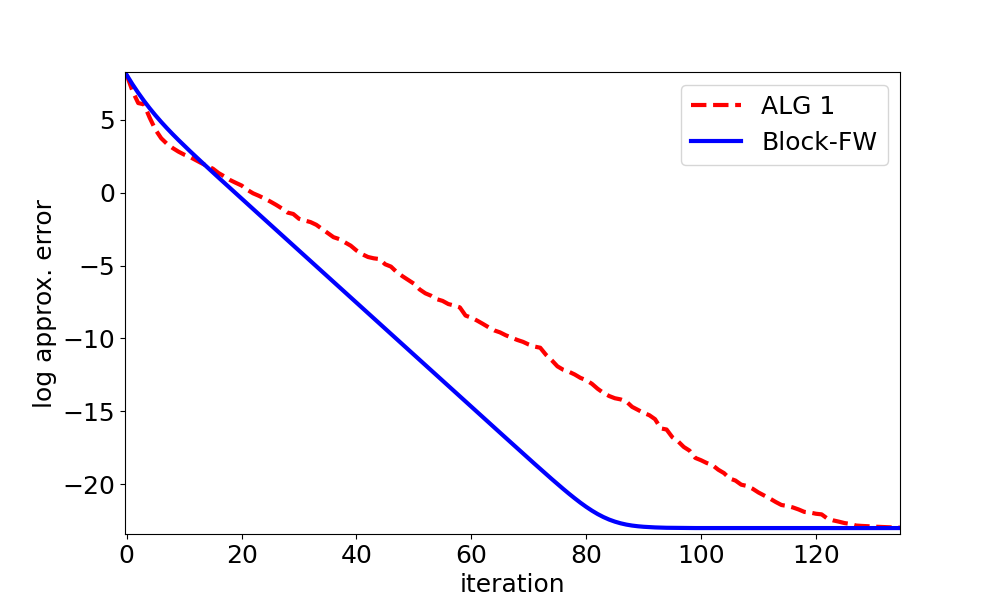}
         \caption{LS, $r^*=5$}
         %\label{fig:y equals x}
     \end{subfigure}
     \begin{subfigure}[t]{0.32\textwidth}
         \centering
         \includegraphics[width=\textwidth]{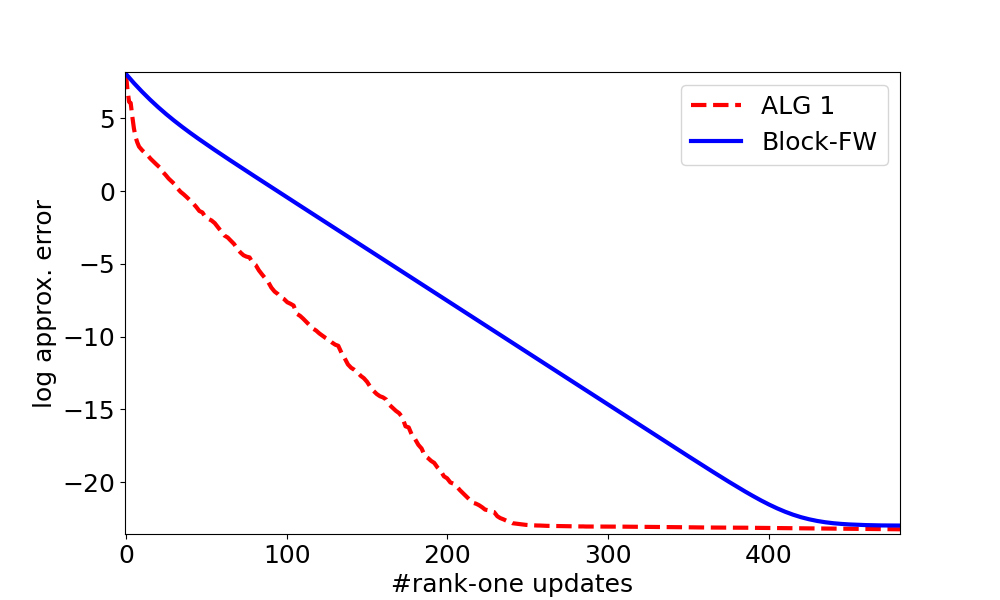}
         \caption{LS, $r^*=5$}
         %\label{fig:three sin x}         
     \end{subfigure}

     \begin{subfigure}[t]{0.32\textwidth}
         \centering
         \includegraphics[width=\textwidth]{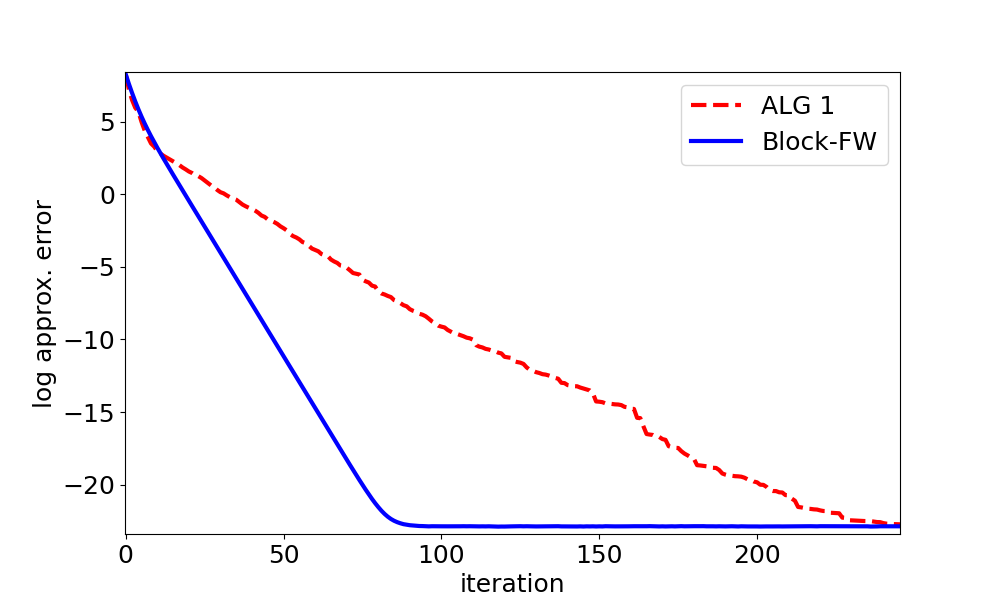}
         \caption{LS, $r^*=7$}
         %\label{fig:y equals x}
     \end{subfigure}
     \begin{subfigure}[t]{0.32\textwidth}
         \centering
         \includegraphics[width=\textwidth]{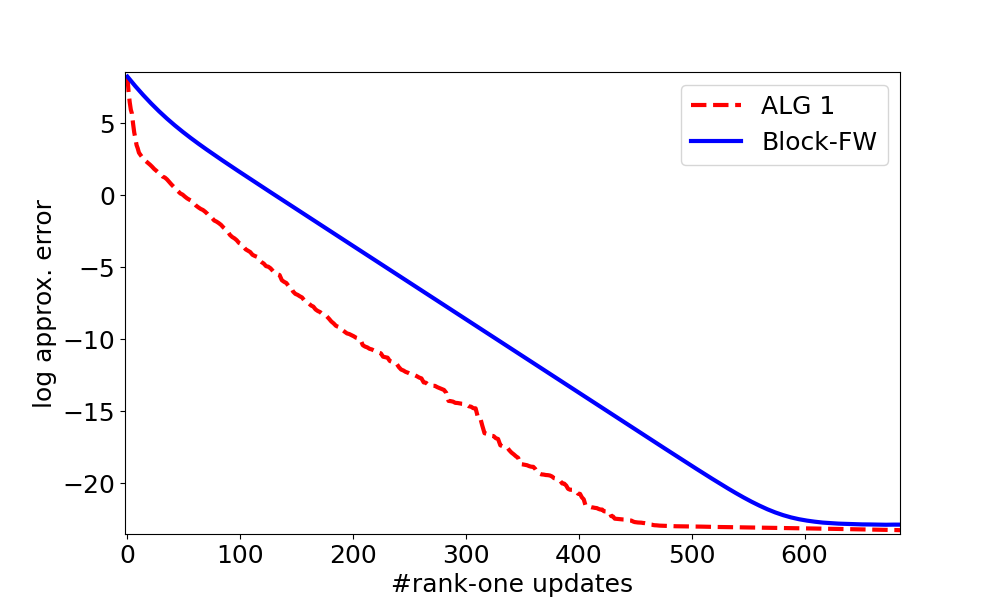}
         \caption{LS, $r^*=7$}
         %\label{fig:three sin x}         
     \end{subfigure}
     
          \begin{subfigure}[t]{0.32\textwidth}
         \centering
         \includegraphics[width=\textwidth]{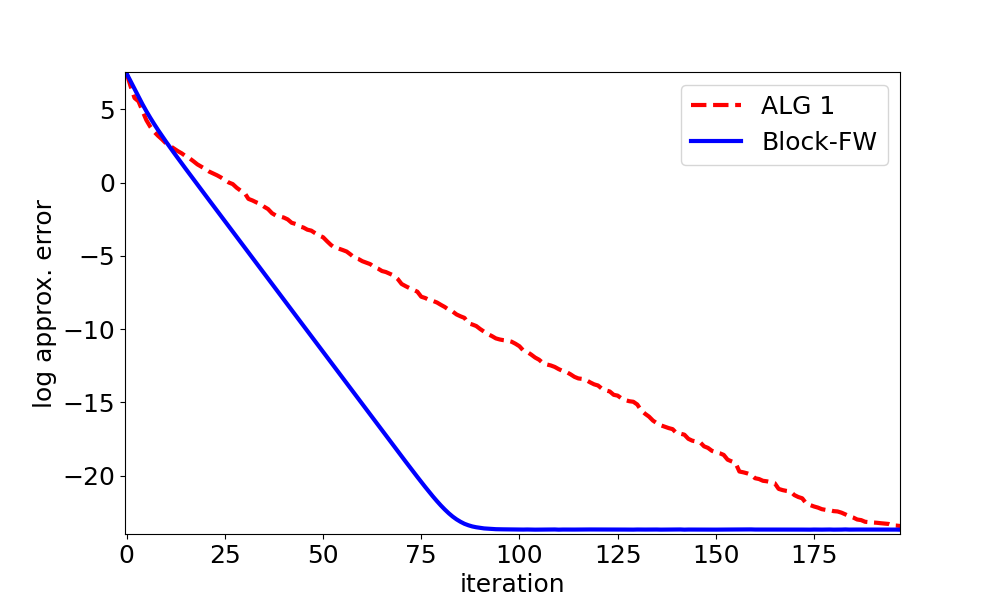}
         \caption{Huber, $r^*=5$}
         %\label{fig:y equals x}
     \end{subfigure}
     \begin{subfigure}[t]{0.32\textwidth}
         \centering
         \includegraphics[width=\textwidth]{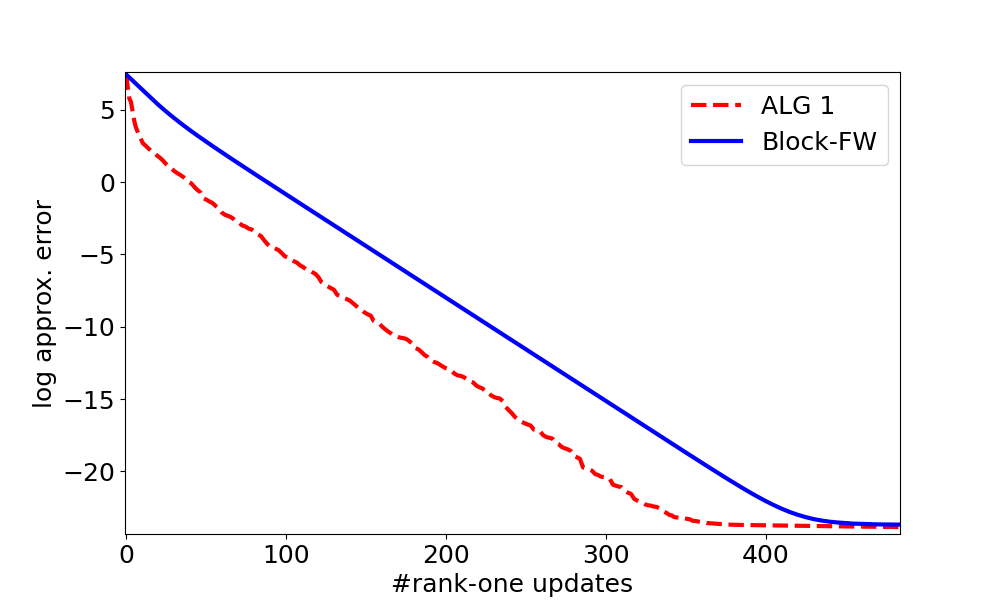}
         \caption{Huber, $r^*=5$}
         %\label{fig:three sin x}         
     \end{subfigure}

          \begin{subfigure}[t]{0.32\textwidth}
         \centering
         \includegraphics[width=\textwidth]{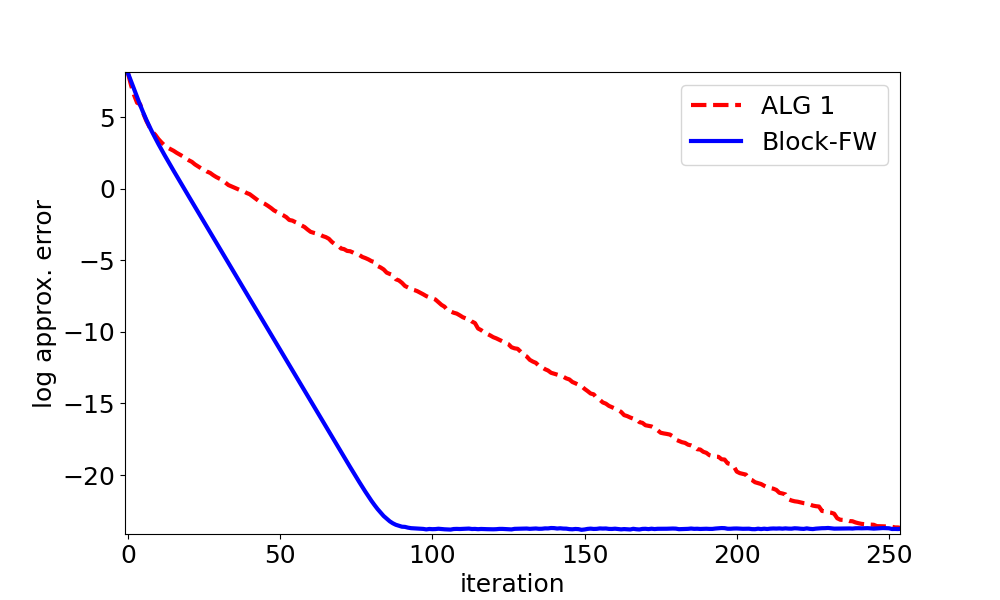}
         \caption{Huber, $r^*=8$}
         %\label{fig:y equals x}
     \end{subfigure}
     \begin{subfigure}[t]{0.32\textwidth}
         \centering
         \includegraphics[width=\textwidth]{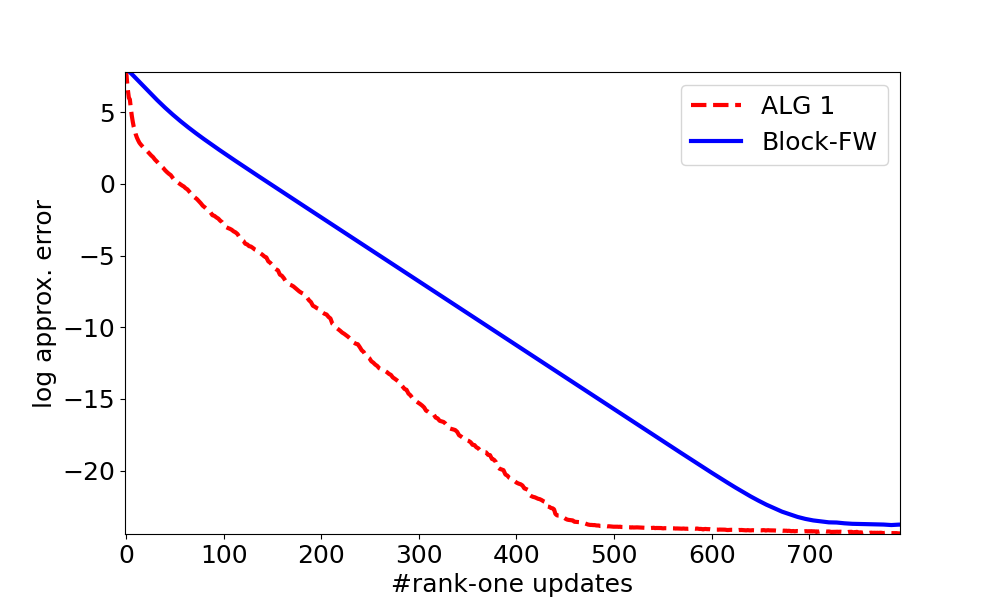}
         \caption{Huber, $r^*=8$}
         %\label{fig:three sin x}         
     \end{subfigure}     
        \caption{Comparison of  the Block-FW method and  our Algorithm  \ref{alg:AFW}. In all experiments, except for Huber $r^*=8$, we set $n=100$, $m=15nr^*$. In Huber $r^*=8$ we set $n=100$, $m=20nr^*$.}
        \label{fig:compareBlockFW}
\end{figure}

%\subsection{What if strict complementarity does not hold?}
%Finally, we examine what happens when we solve Problem \eqref{eq:expProb} when setting $\tau = \trace(\X_{\sharp}) = 1$ and without measurement noise, i.e., $\b = \b_{\sharp}$. In this case, the ground truth matrix $\X_{\sharp}$ is an optimal solution of Problem \eqref{eq:expProb} and in particular $\nabla{}f(\X_{\sharp}) = \mathbf{0}$, meaning the strict complementarity condition does not hold. We see in Figure \ref{fig:noSC} that in this case, our Algorithm \ref{alg:AFW} maintains its linear convergence, while the standard Frank-Wolfe with line-search method converges with a sub-linear rate. It remains an open question if the quadratic growth condition alone suffices to prove linear convergence for our method, see also Section \ref{sec:future}.

%\begin{figure}[H]
%     \centering
%         \includegraphics[width=0.4\textwidth]{FWvsNFW_r=3_noSC}
%        \caption{Comparison of Frank-Wolfe with line-search and our Algorithm \ref{alg:AFW} in case strict complementarity does not hold. We set $n=100$, $r^*=3$, and $m=10nr^*$.}
%        \label{fig:noSC}
%\end{figure}

\section{Future Directions}\label{sec:future}
It is interesting to study whether there exists a first-order method similar to our Algorithm \ref{alg:AFW}, which only relies on a constant number of efficient rank-one matrix computations per iteration, that enjoys any of the following features: I. Converges linearly (in particular without explicit dependence on the ambient dimension) without the strict complementarity assumption (though  still assuming quadratic growth), II. Completely deterministic, and III. Independent of the choice of norm and in particular does not require the knowledge of the smoothness parameter $\beta$.

\section{Acknowledgments}
%Dan Garber is supported by the ISRAEL SCIENCE FOUNDATION (grant No. 2267/22). 
This work was funded by the European Union (ERC,  ProFreeOpt, 101170791). Views and opinions expressed are however those of the author(s) only and do not necessarily reflect those of the European Union or the European Research Council Executive Agency. Neither the European Union nor the granting authority can be held responsible for them.
%There are no competing interests or compliance or ethical issues that we are aware of.

\bibliography{bibs}

@article{meyer1973generalized,
  title={Generalized inversion of modified matrices},
  author={Meyer, Jr, Carl D},
  journal={Siam journal on applied mathematics},
  volume={24},
  number={3},
  pages={315--323},
  year={1973},
  publisher={SIAM}
}

@article{guelat1986some,
  title={Some comments on Wolfe's ‘away step’},
  author={Gu{\'e}lat, Jacques and Marcotte, Patrice},
  journal={Mathematical Programming},
  volume={35},
  number={1},
  pages={110--119},
  year={1986},
  publisher={Springer}
}

@article{garber2023linear,
  title={Linear convergence of Frank--Wolfe for rank-one matrix recovery without strong convexity},
  author={Garber, Dan},
  journal={Mathematical Programming},
  volume={199},
  number={1},
  pages={87--121},
  year={2023},
  publisher={Springer}
}

@inproceedings{ding2020spectral,
  title={Spectral frank-wolfe algorithm: Strict complementarity and linear convergence},
  author={Ding, Lijun and Fei, Yingjie and Xu, Qiantong and Yang, Chengrun},
  booktitle={International conference on machine learning},
  pages={2535--2544},
  year={2020},
  organization={PMLR}
}

@article{allen2017linear,
  title={Linear convergence of a frank-wolfe type algorithm over trace-norm balls},
  author={Allen-Zhu, Zeyuan and Hazan, Elad and Hu, Wei and Li, Yuanzhi},
  journal={Advances in neural information processing systems},
  volume={30},
  year={2017}
}

@inproceedings{NIPS2016_df877f38,
 author = {Garber, Dan},
 booktitle = {Advances in Neural Information Processing Systems},
 editor = {D. Lee and M. Sugiyama and U. Luxburg and I. Guyon and R. Garnett},
 pages = {},
 publisher = {Curran Associates, Inc.},
 title = {Faster Projection-free Convex Optimization over the Spectrahedron},
 volume = {29},
 year = {2016}
}

@article{levitin1966constrained,
  title={Constrained minimization methods},
  author={Levitin, Evgeny S and Polyak, Boris T},
  journal={USSR Computational mathematics and mathematical physics},
  volume={6},
  number={5},
  pages={1--50},
  year={1966},
  publisher={Elsevier}
}

@article{garber2020revisiting,
  title={Revisiting frank-wolfe for polytopes: Strict complementarity and sparsity},
  author={Garber, Dan},
  journal={Advances in Neural Information Processing Systems},
  volume={33},
  pages={18883--18893},
  year={2020}
}

@article{zhou2017unified,
  title={A unified approach to error bounds for structured convex optimization problems},
  author={Zhou, Zirui and So, Anthony Man-Cho},
  journal={Mathematical Programming},
  volume={165},
  pages={689--728},
  year={2017},
  publisher={Springer}
}

@inproceedings{jaggi2010simple,
  title={A simple algorithm for nuclear norm regularized problems},
  author={Jaggi, Martin and Sulovsk, Marek and others},
  booktitle={Proceedings of the 27th international conference on machine learning (ICML-10)},
  pages={471--478},
  year={2010}
}

@inproceedings{jaggi2013revisiting,
  title={Revisiting Frank-Wolfe: Projection-free sparse convex optimization},
  author={Jaggi, Martin},
  booktitle={International conference on machine learning},
  pages={427--435},
  year={2013},
  organization={PMLR}
}

@inproceedings{hazan2008sparse,
  title={Sparse approximate solutions to semidefinite programs},
  author={Hazan, Elad},
  booktitle={Latin American symposium on theoretical informatics},
  pages={306--316},
  year={2008},
  organization={Springer}
}

@article{frank1956algorithm,
  title={An algorithm for quadratic programming},
  author={Frank, Marguerite and Wolfe, Philip and others},
  journal={Naval research logistics quarterly},
  volume={3},
  number={1-2},
  pages={95--110},
  year={1956},
  publisher={Wiley Subscription Services, Inc., A Wiley Company New York}
}

@article{jaggi2011convex,
  title={Convex optimization without projection steps},
  author={Jaggi, Martin},
  journal={arXiv preprint arXiv:1108.1170},
  year={2011}
}

@article{drusvyatskiy2018error,
  title={Error bounds, quadratic growth, and linear convergence of proximal methods},
  author={Drusvyatskiy, Dmitriy and Lewis, Adrian S},
  journal={Mathematics of Operations Research},
  volume={43},
  number={3},
  pages={919--948},
  year={2018},
  publisher={INFORMS}
}

@article{necoara2019linear,
  title={Linear convergence of first order methods for non-strongly convex optimization},
  author={Necoara, Ion and Nesterov, Yu and Glineur, Francois},
  journal={Mathematical Programming},
  volume={175},
  pages={69--107},
  year={2019},
  publisher={Springer}
}

@article{garber2021convergence,
  title={On the convergence of projected-gradient methods with low-rank projections for smooth convex minimization over trace-norm balls and related problems},
  author={Garber, Dan},
  journal={SIAM Journal on Optimization},
  volume={31},
  number={1},
  pages={727--753},
  year={2021},
  publisher={SIAM}
}

@article{lacoste2015global,
  title={On the global linear convergence of Frank-Wolfe optimization variants},
  author={Lacoste-Julien, Simon and Jaggi, Martin},
  journal={Advances in neural information processing systems},
  volume={28},
  year={2015}
}

@article{wiesel2012geodesic,
  title={Geodesic convexity and covariance estimation},
  author={Wiesel, Ami},
  journal={IEEE transactions on signal processing},
  volume={60},
  number={12},
  pages={6182--6189},
  year={2012},
  publisher={IEEE}
}

@article{danon2022frank,
  title={Frank-Wolfe-based algorithms for approximating Tyler's M-estimator},
  author={Danon, Lior and Garber, Dan},
  journal={Advances in Neural Information Processing Systems},
  volume={35},
  pages={3637--3648},
  year={2022}
}

@article{candes2010matrix,
  title={Matrix completion with noise},
  author={Candes, Emmanuel J and Plan, Yaniv},
  journal={Proceedings of the IEEE},
  volume={98},
  number={6},
  pages={925--936},
  year={2010},
  publisher={IEEE}
}

@article{srebro2004maximum,
  title={Maximum-margin matrix factorization},
  author={Srebro, Nathan and Rennie, Jason and Jaakkola, Tommi},
  journal={Advances in neural information processing systems},
  volume={17},
  year={2004}
}

@article{candes2015phase,
  title={Phase retrieval via matrix completion},
  author={Candes, Emmanuel J and Eldar, Yonina C and Strohmer, Thomas and Voroninski, Vladislav},
  journal={SIAM review},
  volume={57},
  number={2},
  pages={225--251},
  year={2015},
  publisher={SIAM}
}

@article{tropp2015convex,
  title={Convex recovery of a structured signal from independent random linear measurements},
  author={Tropp, Joel A},
  journal={Sampling theory, a renaissance: compressive sensing and other developments},
  pages={67--101},
  year={2015},
  publisher={Springer}
}

@article{candes2011robust,
  title={Robust principal component analysis?},
  author={Cand{\`e}s, Emmanuel J and Li, Xiaodong and Ma, Yi and Wright, John},
  journal={Journal of the ACM (JACM)},
  volume={58},
  number={3},
  pages={1--37},
  year={2011},
  publisher={ACM New York, NY, USA}
}

@article{sun2016robust,
  title={Robust estimation of structured covariance matrix for heavy-tailed elliptical distributions},
  author={Sun, Ying and Babu, Prabhu and Palomar, Daniel P},
  journal={IEEE Transactions on Signal Processing},
  volume={64},
  number={14},
  pages={3576--3590},
  year={2016},
  publisher={IEEE}
}

@inproceedings{yurtsever2019conditional,
  title={A conditional-gradient-based augmented Lagrangian framework},
  author={Yurtsever, Alp and Fercoq, Olivier and Cevher, Volkan},
  booktitle={International Conference on Machine Learning},
  pages={7272--7281},
  year={2019},
  organization={PMLR}
}

@inproceedings{pham2023scalable,
  title={A scalable Frank-Wolfe-based algorithm for the max-cut SDP},
  author={Pham, Chi Bach and Griggs, Wynita and Saunderson, James},
  booktitle={International Conference on Machine Learning},
  pages={27822--27839},
  year={2023},
  organization={PMLR}
}

@inproceedings{gidel2018frank,
  title={Frank-wolfe splitting via augmented lagrangian method},
  author={Gidel, Gauthier and Pedregosa, Fabian and Lacoste-Julien, Simon},
  booktitle={International Conference on Artificial Intelligence and Statistics},
  pages={1456--1465},
  year={2018},
  organization={PMLR}
}

@inproceedings{yurtsever2017sketchy,
  title={Sketchy decisions: Convex low-rank matrix optimization with optimal storage},
  author={Yurtsever, Alp and Udell, Madeleine and Tropp, Joel and Cevher, Volkan},
  booktitle={Artificial intelligence and statistics},
  pages={1188--1196},
  year={2017},
  organization={PMLR}
}

@article{freund2017extended,
  title={An extended Frank--Wolfe method with “in-face” directions, and its application to low-rank matrix completion},
  author={Freund, Robert M and Grigas, Paul and Mazumder, Rahul},
  journal={SIAM Journal on optimization},
  volume={27},
  number={1},
  pages={319--346},
  year={2017},
  publisher={SIAM}
}

@book{beck2017first,
author = {Beck, Amir},
title = {First-Order Methods in Optimization},
publisher = {Society for Industrial and Applied Mathematics},
year = {2017},
address = {Philadelphia, PA},
edition   = {}
}

@book{calafiore2014optimization,
  title={Optimization models},
  author={Calafiore, Giuseppe C and El Ghaoui, Laurent},
  year={2014},
  publisher={Cambridge university press}
}
\bibliographystyle{plain}

\appendix

\section{Proof of Lemma \ref{lem:stepsize}}
We first restate the lemma and then prove it.
\begin{lemma}%\label{lem:stepsize}
Let $\X\in\mbS^n_+$ and $\v\in\textrm{Im}(\X)$, $\v\neq\mathbf{0}$. Consider the matrix $\Y = \X - \lambda\v\v^{\top}$ for some $\lambda \geq 0$. If $\lambda \leq (\v^{\top}\X^{\dagger}\v)^{-1}$, then $\Y\succeq 0$. Moreover, if $\lambda = (\v^{\top}\X^{\dagger}\v)^{-1}$, then $\rank(\Y) = \rank(\X)-1$ and  $\X^{\dagger}\v\in\textrm{Ker}(\Y)$.
%$\A-\lambda\v\v^{\top}\succeq 0$ if and only if $1-\lambda\v^{\top}\A^{\dagger}\v \geq 0$. 
\end{lemma}
\begin{proof}
%Using the assumption that $\lambda \leq (\v^{\top}\X^{\dagger}\v)^{-1}$ and $\X\succeq 0$ we have that,
It holds that,
\begin{align*}
1 - \lambda\v^{\top}\X^{\dagger}\v \geq 0 &\Longrightarrow 1- \lambda_1\left({\lambda{\X^{\dagger}}^{1/2}\v\v^{\top}{\X^{\dagger}}^{1/2}}\right) \geq 0\\
&\Longrightarrow  \lambda_n\left({\I -  \lambda{\X^{\dagger}}^{1/2}\v\v^{\top}{\X^{\dagger}}^{1/2}}\right) \geq 0 \\
&\Longrightarrow  \I -  \lambda{\X^{\dagger}}^{1/2}\v\v^{\top}{\X^{\dagger}}^{1/2} \succeq 0 \\
&\Longrightarrow  \X^{1/2}\left({\I -  \lambda{\X^{\dagger}}^{1/2}\v\v^{\top}{\X^{\dagger}}^{1/2}}\right)\X^{1/2} \succeq 0 \\
&\underset{(a)}{\Longrightarrow}  \X -  \lambda\v\v^{\top}  \succeq 0, 
\end{align*}
where (a) relies on the assumption that $\v\in\textrm{Im}(\X)$ and so $\X^{1/2}{\X^{\dagger}}^{1/2}\v = \v$.

We continue to prove the second part of the lemma. Suppose now that $\lambda = (\v^{\top}\X^{\dagger}\v)^{-1}$. Note that since $\v\in\textrm{Im}(\X)$, we have that $\textrm{Im}(\Y) \subseteq \textrm{Im}(\X)$. Consider now the vector $\u = \X^{\dagger}\v\in\textrm{Im}(\X)$. It holds that,
\begin{align*}
\Y\u = \left({\X-\frac{1}{\v^{\top}\X^{\dagger}\v}\v\v^{\top}}\right)\X^{\dagger}\v = \X\X^{\dagger}\v - \v = \mathbf{0},
\end{align*}
implying that $\u\notin\textrm{Im}(\Y)$ and so, $\rank(\Y) = \rank(\X)-1$.
\end{proof}

\section{Fast Update of the Pseudoinverse Matrix}
The following lemma is based on \cite{meyer1973generalized}. While \cite{meyer1973generalized} considers a more general case of non-square complex matrices, the result for symmetric real matrices can be significantly simplified.
\begin{lemma}\label{lem:pinvUpdate}
Let $\A\in\mbS^n$ and let $\c,\d$ be linearly dependent vectors in $\reals^n$. Define  the following: $\k = \A^{\dagger}\c$ (column vector), $\h = \d^{\top}\A^{\dagger}$ (row vector), $\u = (\I-\A\A^{\dagger})\c$ (column vector),  $\v = \d^{\top}(\I-\A^{\dagger}\A)$ (row vector), and $\zeta = 1 + \d^{\top}\A^{\dagger}\c$ (scalar).  Additionally, if $\zeta \neq 0$, define: $\p = -\frac{\Vert{\k}\Vert_2^2}{\zeta}\v^{\top}-\k$ (column vector), $\q = \frac{\Vert{\v}\Vert_2^2}{\zeta}\A^{\dagger}\k - \h^{\top}$ (column vector), and $\sigma = \Vert{\k}\Vert_2^2\Vert{\v}\Vert_2^2+\zeta^2$. Then, it holds that
\begin{align*}
(\A+\c\d^{\top})^{\dagger} = \left\{ \begin{array}{ll}
         \A^{\dagger}  - \k\u^{\dagger}-\v^{\dagger}\h+\zeta\v^{\dagger}\u^{\dagger} & \mbox{if $\c \notin \textrm{Im}(\A)$};\\
         \A^{\dagger} + \frac{1}{\zeta}\v^{\top}\k^{\top}\A^{\dagger} - \frac{\zeta}{\sigma}\p\q^{\top} & \mbox{if $\c\in\textrm{Im}(\A) \textrm{ and } \zeta \neq 0$};\\         
        \A^{\dagger} - \k\k^{\dagger}\A^{\dagger} - \A^{\dagger}\h^{\dagger}\h + (\k^{\dagger}\A^{\dagger}\h^{\dagger})\k\h & \mbox{if $\c\in\textrm{Im}(\A) \textrm{ and } \zeta = 0$}.\end{array} \right. 
\end{align*}
In particular, given $\A^{\dagger}$, one can compute $(\A+\c\d^{\top})^{\dagger}$ in $O(n^2)$ time.
\end{lemma}

\section{Auxiliary Lemmas}
The following lemma is well-known, but for clarity of presentation we include it and the short proof.
\begin{lemma}\label{lem:gradDist}
Let $\Psi:\E\rightarrow\reals$ be $\beta_{\Psi}$-smooth and convex over $\mK$ --- a convex and compact subset of a Euclidean space $\E$. The gradient $\nabla{}\Psi$ is constant over the set of minimizers $\arg\min_{\z\in\mK}\Psi(\z)$, and for any $\z\in\mK$ it holds that
\begin{align}\label{lem:scPoly:cond}
\Vert{\nabla\Psi(\z) - \nabla\Psi(\z^*)}\Vert^2 \leq \beta_{\Psi}\left({\Psi(\z) - \Psi(\z^*)}\right),
\end{align}
where $\z^*$ is any point in $\arg\min_{\z\in\mK}\Psi(\z)$.
\end{lemma}
\begin{proof}
Since $\Psi(\cdot)$ is smooth we have that,
\begin{align*}
\Vert{\nabla\Psi(\z) - \nabla\Psi(\z^*)}\Vert^2 &\leq \beta_{\Psi}\langle{\z-\z^*, \nabla\Psi(\z) - \nabla\Psi(\z^*)}\rangle  \\
&\leq \beta_{\Psi}\left({\Psi(\z) - \Psi(\z^*)}\right) + \beta_{\Psi}\langle{\z^*-\z,\nabla\Psi(\z^*)}\rangle \\
&\leq \beta_{\Psi}\left({\Psi(\z) - \Psi(\z^*)}\right),
\end{align*}
where the second inequality is due to the convexity of $\Psi$, and the last one is due to the first-order optimality condition.

This proves both parts of the lemma.
\end{proof}

\begin{lemma}\label{lem:alphabeta}
If Assumption \ref{ass:sc} holds with $r^* \geq 2$, then the smoothness and quadratic growth constants satisfy $\alpha \leq \beta$.
\end{lemma}
\begin{proof}
Let $\X^*$ be an optimal solution. Under Assumption \ref{ass:sc}, we have that $\rank(\X^*) = r^*$, and let us write its eigen-decomposition as $\X^* = \U\Lambda\U^{\top}$, where $\U\in\reals^{n\times{}r^*}$ and $\Lambda\in\reals^{r^*\times r^*}$ is diagonal and positive definite. Let $\Lambda'\in\reals^{r^*\times{}r^*}$ be a diagonal positive semidefinite matrix with unit trace such that $\rank(\Lambda') < \rank(\Lambda)$, and consider the matrix $\X = \U\Lambda'\U^{\top}\in\mS^n$. From the first-order optimality condition, see for instance Lemma 5.2 in \cite{garber2021convergence}, it follows that $\langle{\X - \X^*, \nabla{}f(\X^*)}\rangle = 0$. Thus, using the smoothness and quadratic growth of $f$ we have that,
\begin{align*}
f(\X) &\leq f(\X^*) + \langle{\X-\X^*,\nabla{}f(\X^*)}\rangle + \frac{\beta}{2}\Vert{\X-\X^*}\Vert_F^2 \\
&= f(\X^*)  + \frac{\beta}{2}\Vert{\X-\X^*}\Vert_F^2 \\
&\leq f(\X^*)  + \frac{\beta}{\alpha}\left({f(\X) - f(\X^*)}\right).
\end{align*}
Since $\rank(\Lambda') < \rank(\Lambda) = \rank(\X^*)$, we have that, under Assumption \ref{ass:sc}, $\X$ is not an optimal solution, meaning $f(\X) > f(\X^*)$, which completes the proof.
\end{proof}

\end{document}